\documentclass{amsart}
\usepackage{amssymb}
\usepackage{amsmath}
\usepackage{graphicx}
\usepackage[colorlinks,bookmarks,allcolors=blue]{hyperref}
\usepackage{tikz}
\usepackage[all,cmtip]{xy}
\usetikzlibrary{arrows,positioning,shapes}
\usepackage[section]{placeins}
\usepackage{dsfont}
\usepackage[mathscr]{euscript}
\usepackage{enumitem}
\usepackage[all,cmtip]{xy}
\usepackage[margin=1.5in]{geometry}

\newtheorem{lemma}{Lemma}[subsection]
\newtheorem{theorem}[lemma]{Theorem}
\newtheorem{proposition}[lemma]{Proposition}
\newtheorem{corollary}[lemma]{Corollary}

\newtheorem{conjecture}[lemma]{Conjecture}
\newtheorem{introthm}{Theorem}

\theoremstyle{definition}
\newtheorem{example}[lemma]{Example}
\newtheorem{remark}[lemma]{Remark}

\newtheorem{definition}[lemma]{Definition}
\newtheorem{question}[lemma]{Question}

\newtheorem{fact}[lemma]{Fact}

\newtheorem{hypothesis}[lemma]{Hypothesis}

\newcommand{\G}[1][n]{\sym[#1][G]}		
\newcommand{\D}[1][n]{\mathscr{D}_{#1}}		

\newcommand{\Q}{\Pi}			
\newcommand{\NN}{\mathbb{N}}			
\newcommand{\ZZ}{\mathbb{Z}}			
\newcommand{\Z}{\mathbb{Z}}			
\newcommand{\QQ}{\mathbb{Q}}			
\newcommand{\C}{\mathbb{C}}			
\newcommand{\RR}{\mathbb{R}}			
\newcommand{\R}{\mathbb{R}}			
\newcommand{\K}{\Bbbk}

\newcommand{\sym}[1][n]{\mathfrak{S}_{#1}}	
\newcommand{\F}{S}			

\newcommand{\FI}[1][G]{\mathtt{FI}_{#1}}		
\newcommand{\FB}[1][G]{\mathtt{FB}_{#1}}		
\newcommand{\Cat}{\mathtt{C}}			
\newcommand{\Set}{\mathtt{Set}}			

\newcommand{\zero}{\hat{0}}			

\DeclareMathOperator{\WH}{WH}			
\DeclareMathOperator{\Ho}{H}			

\newcommand{\wt}[1]{\widetilde{#1}}		
\newcommand{\parts}[1][\beta]{(\wt{#1},z)}

\newcommand{\into}{\hookrightarrow}		
\newcommand{\onto}{\twoheadrightarrow}  

\DeclareMathOperator{\Conf}{Conf}		
\DeclareMathOperator{\rk}{rk}			
\DeclareMathOperator{\Ind}{Ind}			
\DeclareMathOperator{\Hom}{Hom}
\DeclareMathOperator{\Sing}{Sing}

\newcommand{\ol}[1]{\overline{#1}}		
\DeclareMathOperator{\dist}{dist}
\DeclareMathOperator{\Gen}{Gen}
\DeclareMathOperator{\conv}{Conv}

\title[Representation stability phenomena in orbit configuration spaces]{A generating function approach to new representation stability phenomena in orbit configuration spaces}
\author{Christin Bibby}
\address{Department of Mathematics, University of Michigan, Ann Arbor, MI, USA}
\email{\href{mailto:bibby@umich.edu}{bibby@umich.edu}}
\author[Nir Gadish]{Nir Gadish$^1$}
\address{Department of Mathematics, University of Chicago, Chicago, IL, USA}
\curraddr{Department of Mathematics, MIT, Cambridge, MA, USA}
\email{\href{mailto:ngadish@mit.edu}{ngadish@mit.edu}}
\thanks{$^1$ N.G. is supported by NSF Grant No. DMS-1902762}


\begin{document}

\begin{abstract}
As countless examples show, it can be fruitful to study a sequence of
complicated objects all at once via the formalism of generating functions. 
We apply this point of view to the homology and combinatorics of orbit 
configuration spaces: using the notion of twisted commutative algebras, 
which essentially categorify exponential generating functions.
This idea allows for a factorization
of the orbit configuration space ``generating function'' into an infinite
product, whose terms are surprisingly easy to understand. Beyond the
intrinsic aesthetic of this decomposition and its quantitative consequences,
it reveals a sequence of primary, secondary, and higher representation stability
phenomena. Based on this, we give a simple geometric technique for identifying
new stabilization actions with finiteness properties,
which we use to unify and generalize known stability results.
As a first new application of our methods, we establish secondary and higher stability for configuration spaces on $i$-acyclic spaces. For another application, we describe a natural filtration by which one observes a filtered representation stability phenomenon in configuration spaces on graphs.
\end{abstract}

\maketitle
\section{Introduction}
\stepcounter{subsection} 
Let $X$ be a Hausdorff topological space or a separated scheme over an algebraically closed field -- abbreviate and say that $X$ is a \emph{separated space}. A fundamental topological object attached to $X$ is its ordered configuration space $\Conf^n(X)$ of $n$ distinct points in $X$. Analogously, given a group $G$ acting freely on $X$ one defines an ordered configuration space of $n$ points with distinct orbits in $X$:
\begin{equation*}
\Conf^n_G(X) := \{(x_1,\dots,x_n)\in X^n \mid Gx_i\cap Gx_j=\emptyset \text{ for } i\neq j\}.
\end{equation*}
These spaces simultaneously generalize complements of many subspace arrangements such as the ordinary ordered configuration spaces as well as those associated with root systems of type C, see Example \ref{ex:typeC} below.
The symmetric group $\sym$ acts on $\Conf^n_G(X)$ by permuting the labels, and the group $G$ acts on every coordinate separately. Together these operations give an action of the wreath product group $\G := G^n\rtimes \sym$.

In this paper, we study the linear representations that arise in homology $\G \curvearrowright \Ho_*(\Conf^n_G(X))$ for various $n$.
Explicit calculations quickly become combinatorially challenging, and we address these difficulties by importing generating function methods into topology. A concrete consequence of our approach is that the homology exhibits \emph{representation stability} -- roughly, this stability theory gives a notion for when representations of different groups could be considered to be the same, and under which all homology representations eventually stabilize, see e.g. \cite{SS-TCA} and \cite{CEF}. Together with our detailed analysis of the combinatorics governing these spaces in \cite{BG}, representation stability allows us to constrain the types of representations that might occur in homology, vastly generalizing results previously known only for certain manifolds and revealing many new ways in which stability manifests, see \S\ref{subsec:intro-stability} below. For example, our results pertain to singular spaces including graphs, whose configuration spaces fail to exhibit most common forms of stability.

Generating functions have proved invaluable for understanding and organizing the combinatorics at play. For example, considering the compactly supported Euler characteristics of $\Conf^n_G(X)$, one can show that when a finite group $G$ acts on $X$ freely
\begin{equation}
  \sum_{n=0}^{\infty} \chi_c\left( \Conf^n_G(X) \right)\frac{t^n}{n!} = (1+|G|t)^{\frac{\chi_c(X)}{|G|}}
\end{equation}
and see \cite[\S1.34]{VW} for Vakil-Wood's similar motivic zeta function for unordered configuration spaces. The fractional power above is more naturally expressed using exponentials, 
\begin{equation}\label{eq:euler_product_formula}
    = \exp \left( \frac{\chi_c(X)}{|G|} \log (1+|G|t) \right) = \prod_{i=1}^{\infty} \exp \left( \chi_c(X)(-|G|)^{i-1} \frac{t^i}{i} \right).
\end{equation}
We show below in Theorem \ref{thm:intro-product} that the latter product decomposition already holds at the level of chains approximating the homology, and this forms the basis to our stability analysis.

Lifting the above generating function to spaces, think of the entire sequence of ordered configuration spaces along with their group actions at once, and collect them into a single object: a \emph{topological species} $\Conf^\bullet_G(X)$ -- this is essentially an $\NN$-graded space on which the symmetric group $\sym$ acts in the $n$-th component -- a standard categorification of the exponential generating function (see \S\ref{sec:tca} for details).
A major benefit to this approach, as observed by Petersen \cite{petersen}, is that the species $\Conf^\bullet_G(X)$ admits a coproduct structure, coming from the obvious equivariant inclusions
\[
\Conf^{n+m}_G(X) \into \Conf^n_G(X)\times \Conf^m_G(X)
\]
and this operation could be understood as standing behind their many homological stability phenomena. A species with (co)multiplication is known as a \emph{twisted commutative (co)algebra}, or (co)TCA for short, and Petersen \cite[Lemma 4.2]{petersen} shows that in many cases they give rise to representation stability in the sense of Church-Farb (see \cite{CEF}).

In \S\ref{sec:tca}, we adjust the terminology and promote symmetric group actions into ones of wreath products -- replacing the notion of TCA by the $G$-version which we call a \emph{GTCA}. With this, one can make sense of lifting the infinite product decomposition in \eqref{eq:euler_product_formula} to the level of homology: exponentials here stand for \emph{free} GTCAs. And indeed, such a homological decomposition holds in important special cases, including the linear and toric arrangements associated to root systems of type A, B, C and D, see Example \ref{ex:exact_product_formula}. In general, however, the product decomposition holds only at a finite page of a spectral sequence converging to the Borel-Moore homology. The following theorem is stated in a simplified form for the purpose of introduction; see Theorem \ref{thm:ss_factorization} for a general version.
We assume that homology is taken with coefficients in a Noetherian ring for which the K\"unneth formula holds; see our conventions in \S\ref{conventions} below.
\begin{introthm}[\textbf{Homological Product Decomposition}]\label{thm:intro-product}
Let $X$ be a separated topological space with an action of a finite group $G$ and consider the twisted commutative coalgebra of orbit configuration spaces $\Conf^\bullet_G(X)$.  
There is a spectral sequence of GTCAs converging to $\Ho^{BM}_*(\Conf^\bullet_G(X))$, such that when $G$ acts freely
\begin{equation}
\label{eq:intro-product}
    E^1 \cong \bigotimes_{n=1}^{\infty}\Ind^{\FB}_{G\times \sym}\left( \Ho^{BM}_*(X) \boxtimes \wt\Ho_{n-3}(\ol\Q_n)[n-1] \right).
\end{equation}
Here the operation $\Ind^{\FB}_{G\times \sym}$ takes a representation of $G\times \sym$ and freely generates from it a GTCA. Also $\ol\Q_n$ is the classical partition lattice, with the top and bottom elements removed, and its homology is the homology of its nerve.

Lastly, when $X \cong \RR^d$ possibly with some points removed, the sequence collapses and the product formula already holds in homology. More generally, this happens for all $i$-acyclic spaces (see Corollary \ref{cor:i-acyclic} below).
\end{introthm}
Computing the Euler characteristic of the above expression recovers \eqref{eq:euler_product_formula} exactly. Note also that the special case of a punctured $\RR^2$ already includes the complements of hyperplane arrangements coming from root systems of types A and C in both their linear and toric variants.
In the case of type B/C hyperplane arrangements this coincides with the Whitney homology of a Dowling lattice,
for which Henderson \cite{henderson} observed this structure and used it to compute the Frobenius characteristic of the $\G$-action.
Our more general treatment of orbit configuration spaces, as explained next, also includes the linear and toric arrangements associated to root systems of types B and D.

\subsection{Punctured almost free $G$-spaces}
\label{sec:intro-punctures}
We introduced these orbit configuration spaces under the assumption that the group $G$ acts freely on $X$, but in fact we can loosen this condition to allow the configuration space to inhabit some points at which the action fails to be free, while excluding other positions.

From this point on, we will no longer assume that the action $G\curvearrowright X$ is free, but allow only finitely many exceptions to freeness -- such actions are said to be \emph{almost free}. Given a finite $G$-invariant subset $T\subseteq X$ of excluded positions, define the \textit{$T$-excluded orbit configuration space}
\begin{equation}
\label{eq:orbitconfig}
\Conf_G^n(X,T) := \{(x_1,\dots,x_n)\in X^n \mid \forall(i\neq j)Gx_i\cap Gx_j=\emptyset,\, x_i\notin T\}\subseteq X^n.
\end{equation}
The product decomposition in Theorem \ref{thm:intro-product} still holds in this more general setting, but with an extra factor for each $G$-orbit of $T$ as well as one for each non-free $G$-orbit of $X$; see Theorem \ref{thm:ss_factorization} for the full expression.

We view $\Conf_G^n(X,T)$ as a subspace of $X^n$ rather than of the punctured $(X\setminus T)^n$ to allow the factors in the product decomposition to depend on $\Ho_*^{BM}(X)$ rather than on $\Ho_*^{BM}(X\setminus T)$.
This choice is substantially more interesting combinatorially (see \cite{BG} for a thorough treatment), but it also brings real benefits. First, this allows us to exploit properties of $X$ that get corrupted by puncturing, such as it being affine or projective, thus often simplifying spectral sequence calculations, see Remark \ref{rmk:differentials}.

Second, configuration spaces in punctured linear spaces $\R^d\setminus\{r_1,\dots,r_k\}$ get treated as complements of linear subspace arrangements, thus bringing to bare a vast body of knowledge and explicit formulas, see Corollary \ref{cor:affine}.

Lastly, allowing the case in which $G$ does not act freely on $X$ and varying $T$ lets our setup include all sequences of toric root system arrangements: in Example \ref{ex:typeC} we consider the action of $\Z_2$ on $\C^\times$ by group inversion. This action is not free precisely at the two-torsion points $\{\pm 1 \} \subset \C^\times$, and choosing $T$ to be $\emptyset$, $\{+1\}$, or $\{\pm1\}$ yields the arrangement associated to a root system of type D, B, and C, respectively.

\subsection{Representation stability} \label{subsec:intro-stability}
Unpacking the product decomposition and extracting useful information from it presents a new combinatorial challenge. We face it with the framework of representation stability, which for our purposes could be understood as the representation theory of GTCAs -- mainly handling finite generation and Noetherianity of modules and their representation theoretic properties (see Theorem \ref{thm:primary-multiplicity} for the connection with other standard interpretations of the theory).
The multiplication structure on $\Ho^{BM}_*(\Conf^\bullet_G(X))$ gives rise to stabilization operations of introducing points to configurations, e.g.
\[
\Ho^{BM}_{d}(X)\otimes \Ho^{BM}_{i}(\Conf^n_G(X)) \to \Ho^{BM}_{d+i}(\Conf^{n+1}_G(X))
\]
which Petersen observed to be conjugate to the ordinary forgetful map \[\Conf^{n+1}_G(X) \to \Conf^n_G(X)\] under Poincar\'e duality when $X$ is a connected $d$-manifold.
Stabilization operations thus generate a GTCA, over which $\Ho^{BM}_*(\Conf^\bullet_G(X))$ forms a module, and representation stability is synonymous with having this module be finitely-generated (the reason why this can reasonably be called `stability' will be explained below). Indeed, we show the following in Theorem \ref{thm:primary-finite-generation}.
\begin{introthm}[\textbf{Finite generation in homology}]\label{thm:intro-primary_stability}
Let $X$ be a separated space, endowed with an almost free action of a finite group $G$, and let $T\subset X$ be a finite $G$-invariant subset. Assume $\dim \Ho_*^{BM}(X)<\infty$ and let $H^{BM}_{d}(X)\neq 0$ be the top nonvanishing Borel-Moore homology group.

For every $i\geq 0$, the cross product
\[
\Ho^{BM}_{d}(X)\otimes \Ho^{BM}_{i}(\Conf^k_G(X,T)) \to \Ho^{BM}_{d+i}(\Conf^{k+1}_G(X,T))
\]
presents the sequence of codimension $i$ homologies $\Ho^{BM}_{d\bullet-i}(\Conf^\bullet_G(X,T))$ as a filtered module over the free GTCA generated by $H^{BM}_{d}(X)$.

If $d\geq 2$, then every one of these modules is finitely generated. 
Explicitly, for every $i\geq 0$ there exist finitely many classes
\[
\alpha_1,\ldots,\alpha_k \in \coprod_{n \in \NN} \Ho^{BM}_{dn-i}(\Conf^n_G(X,T))
\]
whose images under repeated multiplication by $\Ho^{BM}_{d}(X)$ generate $\Ho^{BM}_{dm-i}(\Conf^m_G(X,T))$ as a $\G[m]$-representation for all $m\in \NN$.

Otherwise, when $d=1$, the homology is endowed with a natural muiltiplicative `collision filtration' under which every submodule of bounded filtration degree
\[
F_{p}\Ho^{BM}_{d\bullet-i}(\Conf^\bullet_G(X,T))
\]
is finitely generated.
\end{introthm}
This finite generation result vastly extends the known scope of applicability of representation stability: in the early days of this theory Church found a notion of stability for the $\sym$-representations $\Ho^*(\Conf^n(M))$ of connected orientable manifolds of dimension $\geq 2$ \cite{church}, which was later recasted as finite generation of a representation of some category \cite{CEF}.
With this approach, 
Wilson \cite{Wi-FIW}, Kupers--Miller \cite{KM}, and Casto \cite{casto}
studied orbit configuration spaces of free actions and showed that they give rise to similar finitely generated representations of categories, with analogous representation theoretic implications.
More recently, Petersen \cite{petersen} brought the TCA point of view to the study of non-equivariant configurations and extended the finiteness result to general spaces with $\Ho^{BM}_{top}(X)$ of rank $1$ along with further technical assumptions. We see Theorem \ref{thm:intro-primary_stability} above as a conceptual improvement over this body of work in the following ways: 
\begin{itemize}
    \item We discuss \emph{orbit} configuration spaces associated with a $G$-action, perhaps with some points removed and not necessarily with a free action.

    \item While most previous results applied to connected orientable manifolds, our theorem encompasses general spaces that may be singular, disconnected and non-orientable. Only the Borel-Moore homology of the space enters our calculation, and therefore only the proper stable homotopy type.
    Under Poincar\'{e} duality one recovers most known finite-generation results for manifolds (see Remark \ref{rem:duality}).
    
    \item Our analysis elucidates the stability aspects of configuration spaces of graphs (that is, 1-dimensional CW-complexes), which were previously understood to be unstable. While naively the homology does not exhibit stability, it is equipped with a natural filtration by finitely-generated modules. Moreover, unlike all previous work on the subject, only the compactly supported Betti numbers of the graph enter our calculation.
\end{itemize}

Finite generation translates to stability of representations since it demonstrates that all homology groups are subquotients of representations naturally induced from
\[
V\boxtimes \underbrace{\left( H \boxtimes \ldots \boxtimes H\right) }_{k \gg 1 \text{ times}}
\]
where $H = \Ho^{BM}_{d}(X)$ is a fixed $G$-representation and $V$ is drawn from a fixed finite list of $\G$-representations. One then has precise branching rules describing the irreducible decomposition of representations of this form.
The representations stablize in the sense that there is a natural naming scheme on the irreducible representations of the groups $\G$, under which $\Ho^{BM}_{dn-i}(\Conf^n_G(X,T))$ all eventually have the same name (see Theorem \ref{thm:primary-multiplicity}).
For the reader familiar with $\FI$-modules, let us remark that working with GTCAs instead of $\FI$-modules (as in \cite{SS-FIG, KM, casto}) allows us to consider cases in which $H$ above is not the trivial representation.
A more detailed discussion is found in \S\ref{sec:primarystab}.

\subsection{Secondary and higher stability} \label{subsec:intro-secondary}
The finite generation and stability result above is associated only with stabilization by a single term in the infinite product decomposition of Theorem \ref{thm:intro-product}. Introducing the actions of the other terms gives rise to higher-order stabilization. More concretely, 
after understanding the multiplication by $\Ho^{BM}_d(X)$, one may factor it out and inquire as to the remaining classes -- the generators of the modules described in Theorem \ref{thm:intro-primary_stability}. These collections of generators themselves form modules over the GTCAs generated by the remaining stabilization operations, and in some cases these too exhibit stability, namely \emph{secondary stability}.

In \S\ref{subsec:generation_locus} we introduce a new geometric technique for recognizing finitely-generated module structures on a bigraded GTCA, indexed by corners of convex rational polygons. Applying this technique to the product decomposition in Theorem \ref{thm:intro-product} lets us identify a multitude of secondary stabilization operations:
\begin{itemize}
    \item multiplying by increasingly higher codimensional homology $\Ho^{BM}_{d-k}(X)$ -- call this \emph{high dimensional secondary stability}, discussed next in Theorem \ref{thm:intro-secondary_stability};
    \item multiplying by terms of the product decomposition \eqref{eq:intro-product} associated with $n>1$ -- in \S\ref{subsec:orbiting_pair} we identify these as a natural analogue of the Miller-Wilson secondary stability operation of introducing an orbiting pair of points to a configuration \cite{MW}; and
    \item multiplying by low dimensional homology, e.g. $\Ho^{BM}_0(X)$ -- a new sequence of stabilization operations that comes out of our geometric approach and endows the homology with finitely generated module structures, see \S\ref{subsec:new_stability} for details.
\end{itemize}

The caveat is that the product decomposition of Theorem \ref{thm:intro-primary_stability}  applies only to the $E^1$-page in a spectral sequence, and it is difficult in general to control how factoring-out stabilization operations interacts with differentials.
While we expect secondary stability in homology to hold more generally, the most explicit statements we make here are about $i$-acyclic spaces, when all differentials vanish. These include, for example, any space of the form $X\times\R$ and any orientable manifold with trivial cup product on $\Ho^*_c$. In particular the results apply to root system arrangements in their affine and toric variants.
See Definition \ref{def:factor_out_action} on how to extract the generating module, and Remark \ref{rmk:differentials} for a discussion regarding $i$-acyclic spaces.

In high dimensional stabilization for $i$-acyclic spaces, multiplying by $\Ho^{BM}_{d-k}(X)$ for $k=1,2,\ldots$, each operation in turn gives rise to a finitely generated module structure on the module of generators of the previous one. A concrete consequence of this pattern is an increased range of homology generated by stabilization operations, stated in the Theorem \ref{thm:intro-secondary_stability} below.
One can extend this sequence of stabilization operations by using the factors of the product decomposition indexed by $n>1$.
Since the first draft of this paper was released, Ho \cite{Ho} has given a similar sequence of stability operations. He uses a completely different approach which requires rational coefficients, but his results apply to a different generalization of configuration spaces.

\begin{introthm}[\textbf{High dimensional secondary stability}]\label{thm:intro-secondary_stability}
Let $X$ be a separated almost free $G$-space with $\dim \Ho_*^{BM}(X)<\infty$, and let $\Ho^{BM}_{d}(X)\neq 0$ be the top nonvanishing homology group. Assume further that $X$ is $i$-acyclic, i.e. the map $\Ho^*_c(X)\to \Ho^*(X)$ is zero, such as an orientable manifold with trivial cup product on $\Ho^*_c$ or any space of the form $X'\times \RR$. Lastly, pick a finite $G$-invariant subset $T\subset X$.

Fix $k < \frac{d}{2}-1$. Then in the range $j\leq (k+1)n$ the cross products give a surjection
\[
\Ind_{G\times \G[n-1]}^{\G}\left( \bigoplus_{i=0}^k \Ho_{d-i}^{BM}(X)\otimes \Ho^{BM}_{d(n-1)+i-j}(\Conf_G^{n-1}(X,T)) \right)
\onto \Ho_{dn-j}^{BM}(\Conf^n_G(X,T))
\]
Further statements are possible for larger $k$, but we omit them from this discussion.
\end{introthm}
These operations turn out to behave increasingly poorly when approaching the middle homological dimension of $X$, setting the case $d=1$ of graphs as now part of a broad phenomenon -- see Example \ref{ex:high_dimensions} for details. Corollary \ref{cor:i-acyclic} expands on the statement of Theorem \ref{thm:intro-secondary_stability} and includes effective bounds on the generation degrees in special cases. We prove Theorem \ref{thm:intro-secondary_stability} as a corollary in Remark \ref{rmk:high-stability}.
 
\begin{figure}[ht]
\begin{tikzpicture}[scale=.6]
\node at (1.1,7.3) {0};
\node at (3,1) {$\dots$};
\draw[-latex] (0,0)--node[below,pos=.9] {$n$} (3+4/3,0);
\draw[-latex] (0,0)--node[left,pos=.95] {$*$} (0,9);
\foreach \c in {0,1,2,3,4}
{
\draw[thick] (\c/3,0)--(\c,\c+\c);
\draw[-latex,thick,opacity=.3,color=blue] (\c,\c+\c)--(3+\c/3,9);
};
\draw[opacity=0] (4/3,0)--node[below,sloped,opacity=1]{\scriptsize$*=dn-j$}(4,8);
\draw[opacity=0] (0,0)--node[above,sloped,opacity=1]{\scriptsize$*=dn$}(3,9);
\node at (8,4.5) {$\Ho_*^{BM}(\Conf^n_G(X,T))$};
\end{tikzpicture}
\hspace{5mm}
\begin{tikzpicture}[scale=.6]
\node at (1.1,7.3) {0};
\node at (5,1) {$\dots$};
\draw[-latex] (0,0)--node[below,pos=.93]{$n$}(6.5,0);
\draw[-latex] (0,0)--node[left,pos=.95]{$*$}(0,9);
\draw[opacity=0] (0,0)--node[above,sloped,opacity=1]{\scriptsize$*=dn$}(3,9);
\foreach \c in {0,1,2,3,4}
{
\draw[-latex,thick,opacity=.3,color=blue] (\c,\c+\c)--(3+\c/3,9);
};
\foreach \c in {0,1,2,3,4}
{
\draw[thick] (\c/2,0)--(\c,\c);
\draw[-latex,thick,opacity=.5,color=red] (\c,\c)--(4.5+\c/2,9);
};
\draw[opacity=0] (2,0)--node[below,sloped,opacity=1]{\scriptsize$*=(d-1)n-j$}(6.5,9);
\end{tikzpicture}
\caption{Visualizing Theorem \ref{thm:intro-secondary_stability} for  $\Ho^{BM}_*(\Conf^n_G(X,T))$ with $k=0$ (left) and $k=1$ (right). Blue indicates the primary stability range, red indicates the next stability range, and black indicates generators for stabilization by $\Ho_d^{BM}(X)$ (left) or by $\Ho^{BM}_d(X)$ and $\Ho^{BM}_{d-1}(X)$ (right).}
\label{fig:secstab}
\end{figure}
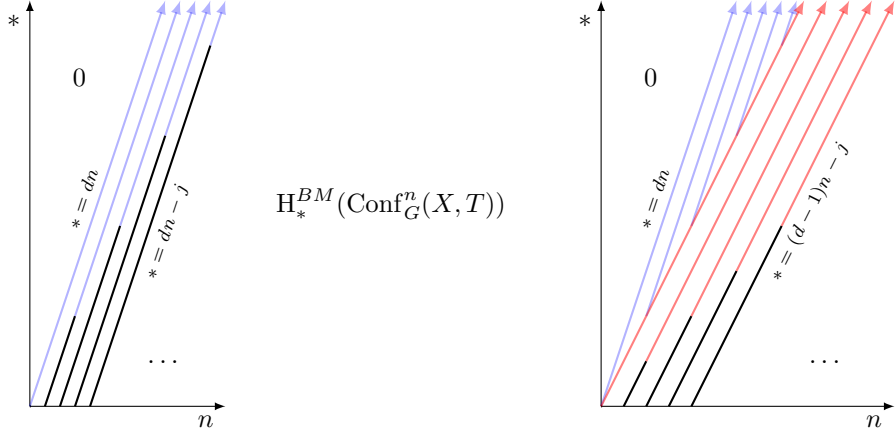

As a particular case of Theorem \ref{thm:intro-secondary_stability}, observe that if the homology is known to vanish in a range: $\Ho^{BM}_{d-1}(X),\ldots, \Ho^{BM}_{d-k}(X)=0$ for $k<\frac{d}{2}-1$, then the primary stability operation will in fact generate $\Ho^{BM}_{dn-j}(\Conf^n_G(X,T))$ in the improved range $j\leq (k+1)n$.
\begin{example}
Fix numbers $k < \frac{d}{2}-1$ and let $M'$ be a $(d-1)$-manifold with finitely generated homology such that $H_1(M')=\ldots=H_k(M')=0$ and an almost free $G$-action. Then for the $d$-manifold $M=M'\times \RR$, the homology $\Ho^{BM}_{dn-j}(\Conf^n_G(M))$ is generated under the $\G$-action by classes of the form $[M]\times \alpha$ for $\alpha\in \Ho^{BM}_{d(n-1)-j}(\Conf^{n-1}_G(M))$ in the range $j\leq (k+1)n$.

Via Poincar\'e duality, this translates to the standard context of cohomological representation stability:
$\Ho^j(\Conf^n_G(M)) \to \Ho^j(\Conf^{n+1}_G(M))$
associated with the maps forgetting a point from a configuration. Here one gets a much improved stable range of $j\leq (k+1)n$ (compare with Church's \cite[Proposition 4.1]{church} and Tosteson's \cite[Examples 1.4 and 1.6]{tosteson}).
\end{example}

\subsection{Conventions}
\label{conventions}
We carry several conventions throughout this paper, which we state here.
The term \textit{separated space} will refer to either
\begin{itemize}
\item a locally compact Hausdorff topological space or 
\item a separated scheme of finite type over some algebraically closed field.
\end{itemize}
A group action on a set $G\curvearrowright X$ is \emph{almost free} if there is some finite subset $S\subseteq X$ for which $G$ acts freely on $X\setminus S$. 
An \emph{almost free} $G$-\textit{space} is a separated space equipped with an almost-free action of a finite group $G$. 

Throughout, we discuss Borel-Moore homology with various coefficients, but we shall typically suppress the coefficients from the notation. One key restriction on coefficient systems is that they satisfy the K\"unneth isomorphism for powers:
\begin{equation}
\label{eq:kunneth}
\Ho^{BM}_n(X\times X^k) \cong \bigoplus_{i=0}^n \Ho^{BM}_i(X)\otimes\Ho^{BM}_{n-i}(X^k),
\end{equation}
for a separated space $X$ and $k\geq 1$. Therefore we consider homology with coefficients in $\K$ that is either
\begin{itemize}
    \item a field,
    \item a Noetherian ring over which $\Ho^{BM}_*(X)$ is a projective module, or
    \item a sheaf of $R$-algebras on $X$ for which $\Ho^{BM}_*(X;\K)$ is projective over a Noetherian ring $R$.
\end{itemize}

Lastly, when $X$ is a scheme and $\K$ is a sheaf of commutative algebras, the Borel-Moore homology $\Ho^{BM}_*(X;\K)$ denotes the \'etale hypercohomology in negative degrees
\[
\Ho^{BM}_*(X;\K) := \mathbb{H}^{-*}_{\acute{e}t}(X;\mathbb{D}\K)
\]
where $\mathbb{D}\K$ denotes the Verdier dual complex.

\subsection{Outline of this paper}
Since this work is directed towards different research communities -- topology, representation stability and algebraic combinatorics -- each comfortable with different common notions, we attempt to make it accessible to all by including proper introduction to the central combinatorial objects and operations as well as explicit examples. The reader is encouraged to skip any and all of those whenever deemed unnecessary.

\S\ref{sec:tca} is a review of basic definitions, examples, constructions, and operations on $G$-twisted commutative algebras (GTCAs).
\S\ref{sec:dowling} contains the main combinatorial and topological inputs, that culminate in the proof of the product decomposition of Theorem \ref{thm:intro-product}.

Lastly, \S\ref{sec:repstability} discusses representation stability: \S\ref{subsec:generation_locus} introduces simple geometric tests for finite generation of a bigraded module over a GTCA -- our primary technique behind our stability results. 
In \S\ref{sec:primarystab} we establish Theorem \ref{thm:intro-primary_stability} on finite generation and multiplicity stability with respect to the primary stabilization operation.
The section also contains our treatment of higher representation stability and ends by identifying new forms of stability that have previously gone unnoticed to the best of our knowledge.

\subsection{Acknowledgements} The authors would like to thank Jeremy Miller and Jenny Wilson for many useful comments.

\section{G-Twisted commutative algebras} \label{sec:tca}
This section recounts definitions and operations on twisted commutative algebras (TCAs), also extending the theory to the equivariant context. We do this at some length since the notion of a TCA is, as of this time, still not standard, but refer the reader to \cite{SS-TCA} as a general reference.
One of the main points is that the collection of all orbit configuration spaces $\Conf_G^n(X,T)$ can be treated as a single mathematical object with an algebraic structure. 
This perspective offers a framework in which orbit configuration spaces and related structures can be described succinctly, which we will later exploit.

Throughout this section, let $G$ be a finite group.

\subsection{Definitions}
Let $\FB$ be the category of all \textbf{f}inite sets $I$ and $G$-\textbf{b}ijections; that is, functions $f:I\overset{\sim}{\to}J$ along with a ``coloring" $g:I\to G$. The composition rule for $I\to J \to K$ is given by first pulling back the $G$-coloring from $J$ to $I$ and then multiplying the two pointwise:
$$
I\overset{(f,g)}{\longrightarrow}J \overset{(f',g')}{\longrightarrow} K \quad \mapsto \quad (f'\circ f, (g'\circ f)\cdot g).
$$
Note that every morphism in $\FB$ is invertible, and the automorphism group of a set $I$ with $|I|=n$ is isomorphic to the wreath product of $G$ with the symmetric group, $G^n\rtimes \sym$. Denote the automorphism group of $I$ by $\G[I]$ and let $[n]= \{1,\ldots,n\}$ for every $n\in \NN$.
\begin{remark}
An equivalent description of $\FB$, or rather an equivalent category, is the category of finite free $G$-sets and $G$-equivariant bijections between them. Sending a finite set $I$ to the free $G$-set $G\times I$ gives an equivalence with the definition above. In particular, one can consider the free $G$-sets $G\times[n]$ and stick to thinking in these natural terms. But below we prefer the definition as stated above since it helps us identify induced representations of this category.
\end{remark}

$\FB$ is symmetric monoidal, with monoidal product given by disjoint union $I\otimes J := I\sqcup J$ under the obvious action on morphisms and monoidal unit being $\emptyset$. This product gives the usual ``block diagonal" embeddings $\G[n]\times\G[m]\hookrightarrow \G[n+m]$. 

In everything that follows, we will be interested in functors $\left[ \FB,\Cat \right]$ into various categories $\Cat$. Now, it is often the case that functor categories between symmetric monoidal categories are themselves symmetric monoidal. Therefore in that context one can discuss algebra objects $A\otimes A\to A$ and modules $A\otimes M \to M$. With $\FB$ these take the following explicit form.
\begin{definition}[\textbf{$G$-twisted commutative algebra, GTCA}]
	A $G$-\emph{twisted commutative algebra} (GTCA for short) in a symmetric monoidal category $(\Cat,\otimes,\mathds{1})$ is a lax monoidal functor $A_\bullet:\FB\to \Cat$ (also denoted $A[\bullet]$). That is, an assignment $A_I$ for every finite set $I$, and for every pair $I$ and $J$ a map
	$$
	A_I \otimes A_J \to A_{I\sqcup J}
	$$
	which is $\G[I]\times \G[J]$-equivariant and satisfies appropriate associativity, symmetry and unit axioms. 
    We shall insist that our GTCAs be unital in the sense that $A_{\emptyset} = \mathds{1}$, the monoidal unit, and the multiplication map $A_{\emptyset}\otimes A_I \to A_I$ is the unit map.

	Analogously, a $G$-\emph{twisted commutative coalgebra} (co-GTCA for short) is a GTCA in the opposite category $\Cat^{op}$. This is a contravariant functor $A^\bullet$ along with compatible structure maps
	$$
	A^{I\sqcup J} \rightarrow A^{I}\otimes A^J.
	$$
\end{definition}

\begin{definition}[\textbf{Modules over GTCAs}]
A module over a GTCA $A_\bullet$ is a functor $M_\bullet: \FB\to \Cat$ equipped with maps
\[
A_I\otimes M_J \to M_{I\sqcup J}
\]
satisfying the predictable axioms.
\end{definition}

\begin{remark}\label{rmk:nskeletal}
The objects of the form $[n]$, where $n\in\NN$, along with $\emptyset$, form a skeletal subcategory of $\FB[G]$. We will often discuss a GTCA $A_\bullet$ (or similarly a module over $A_\bullet$) as a functor on this subcategory, abbreviating \[A_n:=A_{[n]}.\]
In this notation, one can view a GTCA as a graded algebra $A_\bullet\sim\coprod_n A_n$ equipped with an action of $\sym[n][G]$ on $A_n$ for all $n\in \NN$, so that the multiplication is suitably equivariant.
\end{remark}

\begin{remark}[\textbf{Combinatorial species}]
Forgetting the algebra structure, one obtains a well-known object called a combinatorial species. Species were originally studied as a categorification of exponential generating functions, as there are operations on species which correspond to adding, multiplying, and composing generating functions. Species were extended to the equivariant setting by Henderson \cite{henderson}, using an equivalent category to $\FB$ whose objects are finite sets equipped with free $G$-actions and whose morphisms are $G$-equivariant bijections.
Henderson even studied some of the combinatorial species we encounter in this paper, but we study them with their natural algebra structure which we exploit in Section \S\ref{sec:repstability}.
\end{remark}

One of the goals of this section is to give reasonable descriptions of particular GTCAs and their modules, and describe how to ``generate'' them from simpler or finitely many objects.

\begin{definition}[\textbf{Finite generation}]
We say that a GTCA $A_\bullet$ is \emph{finitely generated} if there exists a finite set $S\subseteq \coprod_I A_I$ such that no proper sub-GTCA contains $S$. 
Similarly, a module $M_\bullet$ over a GTCA $A_\bullet$ is \emph{finitely generated} if there exists a finite subset of $\coprod_J M_J$ which is not contained in any proper submodule.
\end{definition}

\subsection{Examples}
This subsection will be devoted to important and relevant examples and constructions of (co-)GTCAs.

\begin{example}[\textbf{The exponential GTCA}] \label{ex:comm}
	Let $(\Cat,\otimes,\mathds{1})$ be a monoidal category, such as topological spaces, posets or sets, all with their Cartesian product. The \emph{exponential} GTCA, denoted by $\mathds{1}_\bullet$, is the one sending every set $I$ to $\mathds{1}$ and every morphism to the identity map. The monoid action is given by the canonical unit morphism $\mathds{1}\otimes \mathds{1} \rightarrow \mathds{1}$.
	
	In the case of a trivial group $G=1$, it is known that modules over $\mathds{1}_\bullet$ are the same thing as $\FI[]$-modules: representations of the category $\FI[]$ of finite sets and all injective maps between them. It is in the context of these objects that Church-Ellenberg-Farb began their work on representation stability. For example, they show in \cite{CEF} that the cohomology of configuration spaces of a connected, orientable manifold of dimension $n\geq 2$ is a finitely generated $\FI[]$-module, and derive explicit representation theoretic conclusions from this fact.
	
	In the case of a general group $G$, a module over $\mathds{1}_\bullet$ is what's known in the literature as an $\FI$-module: representations of the category of finite, free $G$-sets and equivariant injections. In special cases the representation theory of such categories has been studied by Wilson \cite{Wi-FIW}, and the general case was worked out by Sam-Snowden \cite{SS-FIG}. When $G$ is a finite group, $\FI$-modules enjoy a similar theory of representation stability, with numerous applications to both algebra and topology (see the two references mentioned in this paragraph).
\end{example}

\begin{example}\label{ex:starcheck}
In the monoidal category $(\Set,\times,*)$ we have the exponential GTCA $(*)_\bullet$ as in Example \ref{ex:comm}. A related GTCA of sets, denoted by $(\check{*})_\bullet$, is obtained in the same way with the exception of having the value $\emptyset$ in degree 1. 
This GTCA will be a key ingredient for the GTCA we will analyze in \S\ref{sec:dowling}.
\end{example}

The following set of examples make up our motivation to consider GTCAs in the first place.
\begin{example}[\textbf{The power co-GTCA}]\label{ex:powers}
	Given a space $X$ with an action of $G$, the \emph{power co-GTCA} of $X$ is
	\[
	X^\bullet: I\mapsto X^I \cong \underset{|I| \text{ many times}}{X \times \ldots \times X} \quad,\qquad X^\emptyset = \{*\}.
	\]
Explicitly, a point $x\in X^I$ is a function $x:I\to X$, and a morphism $(f,g):I\to J$ acts on $y\in X^J$ by sending it to
	\[
	(f,g).y:= g^{-1}.f^*(y): a\mapsto g(a)^{-1}.y(f(a)).
	\]
	Two sets $I$ and $J$ give a canonical isomorphism
	$$
	X^{I\sqcup J} \overset{\sim}{\rightarrow} X^I\times X^J
	$$
	satisfying all compatibility axioms. Clearly, the same definition works with $X$ in any symmetric monoidal category. Furthermore, note that since all maps in this example are isomorphisms, they can be inverted to get a GTCA instead.
\end{example}

Next we observe that orbit configuration spaces form a co-GTCA as well.
\begin{example}[\textbf{Orbit configuration spaces}] \label{ex:configs}
	Recall the spaces $\Conf_G^n(X,T)$ defined in \eqref{eq:orbitconfig} in \S\ref{sec:intro-punctures}. As we let $n$ vary and range over all finite sets, the collection of these spaces (for a fixed $X$ and $T$) has the structure of a co-GTCA: it can be described as the maximal sub-co-GTCA $\Conf_G^\bullet \subset X^\bullet$ in which $\Conf_G^{[1]}\subset X$ is the complement of $T$ and $\Conf_G^{[2]}\subset X^2$ is disjoint from the diagonal $\Delta\subset X^2$.
	To see that the collection indeed forms a sub-co-GTCA one only needs to observe $\Conf_G^{[n+m]}$ is identified with an open subspace of $\Conf_G^{[n]}\times \Conf_G^{[m]}$ inside $X^{n+m}$, given by imposing more inequalities of the form $x_i\neq g.x_j$.
\end{example}

\begin{example}[\textbf{Borel-Moore homology}]\label{ex:BM}
Consider a co-GTCA $A^\bullet$ of topological spaces for which the maps $A^{I\sqcup J}\to A^I\times A^J$ are open inclusions. One such example is the co-GTCA of orbit configuration spaces in Example \ref{ex:configs}.
Since Borel-Moore homology (with coefficients as in Conventions \ref{conventions}) is a contravariant functor on open inclusions, $\Ho_*^{BM}(A^\bullet)$ is a GTCA of graded $\K$-modules. The GTCA $\Ho_*^{BM}(\Conf_G^\bullet(X,T))$ will be the primary focus of our study in later sections.
\end{example}

Before we proceed to our main examples, we describe a number of constructions and operations on GTCAs that make the treatment of them more formal and streamlines the pursuing discussions. These are all special cases of the left Kan extension, and they all work as expected. We include a detailed account for completeness.

\subsection{Induced GTCAs}
The power GTCA from Example \ref{ex:powers} above is a special case of induced GTCAs, which are central to our study of the combinatorics of orbit configuration spaces.

	Fix some $n\in \NN$. Then the restriction
\[A_\bullet \mapsto A_n
\]
associates to every GTCA $A_\bullet$ an object with an action of the wreath product group $\G=G^n\rtimes\sym$. This restriction functor often has a left adjoint -- the induction, denoted by $\Ind^{\FB}_{(n)}$. We construct this induction explicitly in the case of $\K$-modules. The same construction works for sets, posets, topological spaces and so on.

Let $V$ be a representation of $\G$. Then the tensor powers $V^{\otimes \bullet}$ naturally form a $\G$-TCA (with base group $\G$ instead of $G$). The category $\FB[\G]$ is a monoidal subcategory of $\FB[G]$ under the functor 
\begin{equation}\label{eq:spacingfunctor}
    \iota_{\times n}:[k]\mapsto [k\times n].
\end{equation}
The left Kan extension along $\iota_{\times n}$ acquires a GTCA structure and is the sought after induced module $\Ind^{\FB}_{(n)}(V)_\bullet$. More explicitly, the wreath product $\sym[k][\G]=\sym[k][\sym[n][G]]$ is a subgroup of $\G[kn]=\sym[kn][G]$ and we have
\begin{equation}
\Ind^{\FB}_{(n)}(V)_{kn}=\Ind_{\sym[k][\G]}^{\G[kn]} V^{\otimes k}
\end{equation}
while $\Ind^{\FB}_{(n)}(V)_{r} = 0$ unless $n | r$. When working in other categories, the object $0$ must be replaced with the initial object $\emptyset$.

\begin{remark}
\
\begin{enumerate}
\item
In the case where $G$ is trivial and $V$ is a representation of $\sym$, Sam and Snowden \cite{SS-TCA} denote the TCA $\Ind^{\FB[]}_{(n)}(V)_\bullet$ by $\operatorname{Sym}(V)$. The idea is to generate a \textit{free} GTCA, similar to how one generates a polynomial ring, with $V$ in degree $n$.
\item
Recall that a symmetric monoidal category is equipped with isomorphisms $A\otimes B\to B\otimes A$, for any two objects $A$ and $B$. When describing the induction above, the symmetric group  $\sym[k]$ permutes the tensor factors of $V^{\otimes k}$ using this isomorphism. In particular, when working with graded modules, the switching operation involves $(-)$ signs that makes the GTCA a graded-commutative algebra.
\end{enumerate}
\end{remark}

\begin{example}[\textbf{Set partitions}]\label{ex:partitions}
Consider the case where $G=1$ and construct the GTCA (really a TCA) induced by the trivial $\sym$-set $*$. Its value on $[k\times n]$ is the induced set
\[
\Ind^{\sym[k\times n]}_{\sym[k][\sym[n]]}(*) \cong \sym[kn]/\sym[k][\sym].
\]
To identify this set consider the collection of set partitions of $[k\times n]$ into blocks of size $n$. The group $\sym[kn]$ permutes these permutations transitively, and the stabilizer of the partition 
\[
\left\{[1,\ldots,n], [n+1,\ldots, 2n],\ldots,[(k-1)n+1,\ldots,kn]\right\}
\]
is the wreath product $\sym[k][\sym[n]]$. Therefore the induced TCA assigns to a set  all the possible ways to partition it into blocks of size $n$.

When $n=1$, this is simply the exponential TCA from Example \ref{ex:comm}: $\Ind_{(n)}^{\FB[]} * = (*)_\bullet$.
\end{example}

A well-known characterization of induced $G$-representations is by a direct sum decomposition $\oplus _{x\in X} V_x$ on which $G$ acts by permuting the summands transitively. An analogous useful characterization for induced GTCAs is given by the following, whose proof is a straightforward generalization of the case of induced group representations.
    \begin{fact}[\textbf{Induction characterization}] 
    A GTCA $A_\bullet$ is induced from $A_n$ if and only if there exists a GTCA of sets $X_\bullet$, induced from a transitive $\G$-set $X_n$, and the following holds: for all $k\in \NN$ one has 
    $$\displaystyle{A_k = \coprod_{x\in X_k} A_x}$$
for some objects $A_x$, and the category $\FB$ acts on $A_\bullet$ in a way compatible with this decomposition. That is:
    \begin{itemize}
    \item Every morphism $f:[k]\to[k]$ of $\FB$ permutes the summands according to its action on $X_\bullet$, i.e. $f(A_x) = A_{f(x)}$.
    \item The products of $A_\bullet$ define isomorphisms $A_x\otimes A_y \overset{\sim}{\longrightarrow} A_{x\cdot y}$.
    \end{itemize}
    \end{fact}

\subsection{Change of group}
	Another operation that plays an important role in our combinatorial analysis arises when changing the underlying group $G$.	Given a subgroup $H\leq G$, there is a natural inclusion of wreath products $\sym {}[H] \leq \sym{}[G]$ giving rise to a monoidal inclusion functor
    \begin{equation}\label{eq:groupchangefunctor}
    \iota_H^G:\FB[H] \hookrightarrow \FB.
    \end{equation}
Restricting a GTCA along this inclusion yields an HTCA. As with the previous operation, this restriction often has a left adjoint given by the left Kan extension and denoted here by $\Ind_{\FB[H]}^{\FB}$. An explicit description of this extension is given for every $n$ by
\begin{equation} \label{eq:change_group}
\left(\Ind_{\FB[H]}^{\FB}A\right)_n = \Ind_{\sym{}[H]}^{\sym{}[G]} (A_n) \cong \Hom_{\FB}(\iota^G_H [n],[n]) \times_{\sym{}[H]} A_n.
\end{equation}
The product operation for the GTCA $(\Ind_{\FB[H]}^{\FB} A)_\bullet$ is the obvious one given by $\coprod$ on the left factor along with the product of the HTCA $A_\bullet$ on the right factor.

\begin{example}[\textbf{Colorings}]\label{ex:colorings}
Let $H\leq G$ be a fixed subgroup. The set of $(G/H)$-colorings on $[n]$, i.e. the set of functions $[n]\to G/H$, admits a transitive action of the wreath product group $\sym{}[G]$, and the stabilizer of the constant function $k\mapsto 1H$ is the wreath product group $\sym{}[H]$. Therefore the set of all such colorings on $[n]$ is the induced set $\Ind^{\sym{}[G]}_{\sym{}[H]}(*)$.

Now consider the exponential HTCA of sets $(*)_\bullet$ from Example \ref{ex:comm}. The induced GTCA $\Ind_{\FB[H]}^{\FB}(*)_\bullet$ sends a finite set $I$ to the set of its $(G/H)$-colorings.
\end{example}

This group-change operation interacts well with the induction operation from the previous section as well as with disjoint unions. For example, if $V$ is an $H$-representation and $\Ind_H^G(V)$ is its (ordinary) induced $G$-representation, there is an obvious isomorphism
\begin{equation*}
\Ind_{\FB[H]}^{\FB} \Ind^{\FB[H]}_{(1)}(V) \cong \Ind^{\FB}_{(1)} \Ind_H^G(V).
\end{equation*}
We shall abbreviate this composition of inductions to $\Ind_H^{\FB}$.

\begin{remark} \label{rm:group_change_with_spacing}
For a subgroup $H\leq \G$, the group change operation  \eqref{eq:groupchangefunctor} can be composed with the spacing operation \eqref{eq:spacingfunctor}:
\begin{equation*}
\FB[H]\overset{\iota_H^{\G}}{\hookrightarrow} \FB[\G]\overset{\iota_{\times n}}{\hookrightarrow} \FB.
\end{equation*}
The resulting induction produces GTCAs taking non-trivial values only in degrees $k\times n$ with $k\geq 1$. We will denote this induction operation by $\Ind_{\FB[H](n)}^{\FB[G]}$. The collection of such operations satisfies the relations
\begin{equation*}
\Ind^{\FB}_{\FB[H](n)}\Ind^{\FB[H]}_{\FB[K](m)} = \Ind^{\FB}_{\FB[K](nm)} \quad \text{ for } K\leq \sym[m][H] \text{ and } H \leq \sym[n][G].
\end{equation*}
As above, we abbreviate the induction of $H$-representations to GTCAs
\begin{equation}
\Ind_H^{\FB} := \Ind^{\FB}_{\FB[H](n)} \Ind^{\FB[H]}_{(1)}.
\end{equation} 
\end{remark}

\begin{example}\label{ex:partitionGTCA}
A key example is $\Ind_{G\times\sym}^{\FB}(*)$, where
$H=G\times\sym$ sits diagonally inside $\G=G^n\rtimes\sym$ and acts trivially on the singleton set $*$.
When $G$ is trivial, we simply obtain partitions with blocks of size $n$ as in Example \ref{ex:partitions}. 
For general $G$, this combines Examples \ref{ex:partitions} and \ref{ex:colorings}, and we obtain partitions whose blocks are of size $n$ and  equipped with $G$-colorings (up to equivalence). We will more formally define these $G$-partitions in \S\ref{sec:dowling}.

Let us consider the case $G=\Z_2$, which we view as $\Z_2=\{+,-\}$. Here, we obtain \textit{signed} partitions with blocks of size $n$. That is, the elements in a block are colored with either $+$ or $-$, considered up to multiplication by $-$ throughout the block.
\end{example}

\subsection{Products of GTCAs} \label{sec:product_GTCA}
Let $A_\bullet$ and $B_\bullet$ be two GTCAs in $(\Cat,\otimes,\mathds{1})$.
Define their product $(A\otimes B)_\bullet$ by the Day convolution:
\begin{align}
(A\otimes B)_{[n]} & \cong  \bigoplus_{i+j=n} \Ind_{\G[i]\times \G[j]}^{\G} \left(A_{[i]}\otimes B_{[j]}\right)  \\
&\cong  \bigoplus_{i+j=n} \Hom_{\FB}\left([i]\coprod[j],[n]\right)\underset{{\small \G[i]\times\G[j]}}{\times} \left(A_{[i]}\otimes B_{[j]}\right) \nonumber
\end{align}
Note that this is precisely the left Kan extension along the disjoint union functor $\coprod: \FB^2  \to \FB$, and therefore it commutes with other induction operations in the obvious way:
\[
\Ind_{\FB[H]}^{\FB} (A\otimes_{\FB[H]} B) \cong \Ind_{\FB[H]}^{\FB}(A)\otimes_{\FB} \Ind_{\FB[H]}^{\FB}(B).
\]

\begin{example}[\textbf{Set partitions}]\label{ex:products}
Take $G=1$ and let $P^n_\bullet$ and $P^m_\bullet$ be the TCAs in $(\Set,\times,*)$ from Example \ref{ex:partitions}, assigning to a set  its collection of partitions into blocks of equal size $n$ or $m$ respectively. Then if $n\neq m$, the product TCA
\[
P^{\{n,m\}} := P^n \times P^m
\]
assigns to a set the collection of partitions into blocks of size either $n$ or $m$.

Multiplying further gives TCAs $P^{S}_\bullet$ for any set $S$ of natural numbers, parameterizing partitions with block sizes in $S$. Then whenever $S\cap T =\emptyset$, the product of those gives
\[
P^S \times P^T \cong P^{S\cup T}.
\]
\end{example}

\begin{remark}[\textbf{Exponential property of induction}]
Observe that the induction operation behaves like an exponential with respect to the tensor product of GTCAs, in the sense that
\[
\Ind^{\FB}_{(1)}(V\oplus W) = \Ind^{\FB}_{(1)}(V)\otimes \Ind^{\FB}_{(1)}(W).
\]
This is simply a reflection of the expansion
\[
(V\oplus W)^{\otimes n} = \bigoplus_{i+j=n} \Ind^{\sym}_{\sym[i]\times\sym[j]} V^{\otimes i}\otimes W^{\otimes j} = \bigoplus_{i+j=n}\Ind^{\G}_{\G[i]\times\G[j]} V^{\otimes i}\otimes W^{\otimes j}.
\]
Analogous formulas apply when working in other categories, like spaces or posets.
This should not be surprising to the reader familiar with species: this induction essentially performs a plethystic product (also known as substitution or composition) with the exponential TCA.
\end{remark}

\section{Dowling posets and orbit configuration spaces}\label{sec:dowling}
The combinatorial object that arises in the calculation of the homology of an orbit configuration space from the inclusion $\Conf_G^n(X,T)\subseteq X^n$ is the intersection data of the components of $X^n\setminus\Conf_G^n(X,T)$. 
This intersection data corresponds to a stratification of $X^n$.
Thus, the Borel--Moore homology of $\Conf^n_G(X,T)$ can be computed using the spectral sequence for stratified spaces from \cite{petersen}, whose purpose is to separate the topology of $X$ from the combinatorics of the arrangement $X^n\setminus \Conf_G^n(X,T)$.

In \cite{BG}, we introduced and utilized the poset of labeled partial $G$-partitions $\D^T(G,\F)$ as a combinatorial model for the strata (see \cite[Section 3.4]{BG} for definitions and a general discussion). We also noted in \cite[Proposition 2.1]{BG} that these posets are in fact functorial in $n$ and monoidal with respect to the disjoint union. Using the present terminology, this monoidality is rephrased to say that the collection $\D[\bullet]^T(G,\F)$ forms a GTCA. Its structure as such is the subject of this section.

\subsection{Dowling posets}

First, we remind the reader of some relevant notation, used to model the intersection pattern of strata in $X^n$, but warn that without the context of \cite{BG} the following definition will likely be opaque. 
Throughout this section, we let $G$ be a finite group.

\begin{definition}[\textbf{Partial $G$-partitions}] Let $I$ be any finite set. A \textbf{partial $G$-partition} of $I$ is a collection 
$\wt{\beta}=\{\wt{B_1},\dots,\wt{B_\ell}\}$ consisting of a partition
$\beta=\{B_1,\dots,B_\ell\}$ of a subset $\cup B_i\subseteq I$ along with a \emph{relative $G$-coloring} on each block:
functions $b_i:B_i\to G$, defined up to the equivalence relation $b_i\sim b_i g$ for $g\in G$. 
The \textbf{zero block} of a partial $G$-partition $\wt{\beta}$ of $I$ 
is the set 
$Z:=I \setminus\cup_{B\in\beta} B$.
\end{definition}
In the current context we shall be interested in the following refinement of the above notion of partial $G$-partition, as it arises from the combinatorics of orbit configuration spaces.

\begin{definition}[\textbf{Dowling posets}]
Fix two finite $G$-sets $\F$ and $T$. For a finite set $I$, let $\D[I]^T(G,\F)$ be the set of pairs $\parts$, where $\wt{\beta}$ is a partial $G$-partition of $I$ and $z:Z\to \F$ is a coloring of its zero block, with the restriction that  $|z^{-1}(G.s)|=1$ only if $s\in T$.

The set $\D[I]^T(G,\F)$ is partially ordered via the following covering relations:
\begin{description}
\item[(\textbf{merge})]
$(\wt{\beta}\cup\{\wt{A},\wt{B}\},z)\prec
(\wt{\beta}\cup\{\wt{C}\},z)$
where $C=A\cup B$ with $c\sim a\cup bg$ for some $g\in G$, and
\item[(\textbf{color})]
$(\wt{\beta}\cup\{\wt{B}\},z) \prec (\wt{\beta},z')$
where $z'$ is the extension of $z$ to 
$Z' = B\cup Z$ given on $B$ by a composition
\[
{B} \overset{b}{\rightarrow} G \overset{f}{\rightarrow} S
\]
for some $G$--equivariant function $f$.
\end{description}
\end{definition}

Note that $\D[I]^T(G,S)=\D[I]^{T\cap S}(G,S)$, so we may and will assume that $T\subseteq S$.
When $I=[n]$, denote $\D[{[n]}]^T(G,S)=\D[n]^T(G,S).$

Two notable special cases of these posets are: 
\begin{itemize}
\item
the partition lattice $\Q_I$, which is the case $G=\{1\}$ and $\F=\emptyset$; and 
\item 
the Dowling lattice $\D[I](G)$, which is the case $\F=T=\{*\}$ is a single point. 
\end{itemize}
Note that for any $T$, $\D[I]^T(G,S)$ is a subposet of $\D[I]^S(G,S)$; of particular interest is $\D[I]^\vee(G):=\D[I]^\emptyset(G,\{*\})$ which is a sublattice of the Dowling lattice $\D[I](G)=\D[I]^{\{*\}}(G,\{*\})$ consisting of partial $G$-partitions with nonsingleton zero block.
The following is an example which is not a lattice.

\begin{example}[\textbf{Toric type C poset}]\label{ex:typeC}
Let $G=\ZZ_2$ and $S=\{+,-\}$ with $G$ acting on $S$ trivially.
$\D^{\pm}(G,\pm)$ is not a lattice for any $n\geq2$, e.g. because there is not a unique maximum element: the maximal elements are all the different $S$-colorings of the set $\{1,2,\dots,n\}$.

The poset $\D^{\pm}(\ZZ_2,\pm)$ is discussed in \cite[Ex. 2.2.2.]{BG}, and it arises from the toric arrangement associated to the type C root system. The toric arrangements associated to the type B and D root systems give rise to the subposets $\D^{+}(\ZZ_2,\pm)$ and $\D^\emptyset(\ZZ_2,\pm)$, respectively.
\end{example}

\subsection{The GTCA of Dowling posets}

Properties of the Dowling posets $\D[I]^T(G,S)$ can be neatly expressed in the language of GTCAs, which we will explain next.

Each Dowling poset $\D[I]^T(G,\F)$ admits a natural action of a wreath product group $\G[I]$. Furthermore, the disjoint union operation on sets clearly extends to a product operation on the collection of partial $G$-partitions:
\[
\D[I]^T(G,\F)\times \D[J]^T(G,\F) \to \D[I\sqcup J]^T(G,\F),
\]
compatible with the actions on the two sides. 
Therefore the collection $\D[\bullet]^T(G,\F)$ has the structure of a GTCA, taking values in the category of posets.

Since a partition is trivially a disjoint union of its blocks, this GTCA decomposes as a product once one forgets the ordering. This is not surprising in light of Examples \ref{ex:partitions}, \ref{ex:colorings} and \ref{ex:products} from the previous section, and is readily observed in the language of species (as Henderson observes in the case $S=\{0\}$ \cite{henderson}). We state this formally in the following lemma.

\begin{lemma}[\textbf{Dowling GTCA factorization}] \label{lem:poset_factorization}
Let $G$ be a finite group, and let $\F$ and $T$ be two finite $G$-sets. Pick orbit representatives $s\in [s]$ for all $[s]\in \F/G$ and let $G_s$ be the corresponding stabilizer in $G$. Then as a GTCA of \underline{sets} $\D[\bullet]^T(G,\F)$ factors as
\begin{equation*}
\D[]^T(G,\F) \underset{\texttt{Set}}{\cong}
\prod_{n=1}^{\infty} \Ind^{\FB}_{G\times \sym}(*) \times \prod_{[s]\in \F/G} \Ind_{\FB[G_s]}^{\FB}(*)_{\bullet}^T(s)
\end{equation*}
where $(*)_{\bullet}^T(s)$ denotes $(*)_{\bullet}$ when $s\in T$ and $(\check{*})_{\bullet}$ when $s\notin T$.

The factors were described in Examples \ref{ex:comm}, \ref{ex:starcheck}, and \ref{ex:partitionGTCA}.
\end{lemma}
\begin{proof}
Write $\D[\bullet]$ for $\D[\bullet]^T(G,\F)$. The image in $\D[\bullet]$ of $*$ from each induction in the product is as follows:
\begin{enumerate}
\item Sending $*$ to the block $[1,\ldots,n] \in \D[{[n]}]$ with each element colored by the same element of $G$ defines a map $\Ind^{\FB}_{G\times \sym}(*)\into \D[\bullet]$ whose image is exactly the $G$-partitions with blocks of size $n$.
\item When $s\in T$, sending $*$ in degree $n\geq 1$ to the  zero block $[1,\dots,n]$ with constant coloring by $s$ defines a map $\Ind^{\FB}_{\FB[G_s]}(*)_\bullet\into \D[\bullet]$ whose image is exactly the zero blocks with coloring in the orbit $[s]$.
\item Similarly to the previous case, when $s\in S\setminus T$, sending $*$ in degree $n>1$ to the zero block $[1,\ldots,n]$ with constant coloring by $s$ defines a map $\Ind^{\FB}_{\FB[G_s]}(\check{*})_{\bullet}\into \D[\bullet]$ whose image is exactly the zero blocks of size $>1$ and coloring in the orbit $[s]$.
\end{enumerate}
The product of these sub-GTCAs consists precisely of all possible ways to construct a partial $G$-partition of a set with an $S$-colored zero block uniquely.
\end{proof}

The above product decomposition seems to require forgetting the order relation, and thus we pose the following question.
\begin{question}
Is it possible to define a product operation on GTCAs of posets so that the decomposition of Lemma \ref{lem:poset_factorization} holds in this category? 
\end{question}

\subsection{Homology of Dowling lattices}\label{sec:poset_homology}

The GTCA  $\D[\bullet](G)=\D[\bullet]^{\{*\}}(G,\{*\})$ of Dowling lattices gives rise to a fascinating algebraic GTCA given by their top homology.

The order complex attaches a topological space to every poset in a functorial way. In the case of Dowling lattices $\D(G)$ it is well-known that, after removing the top and bottom elements, the order complex has the homotopy type of a wedge of $(n-2)$-dimensional spheres (see \cite{folkman}). Thus the main topological invariant is the $\G$-representation on $\Ho_{n-2}(\overline{\D(G)})$, where we assume coefficients are as in Conventions \ref{conventions}.

Letting $n$ vary, one can promote this sequence of representations to a GTCA as follows. For every lattice $P$ let $\ol{P}$ denote the result of removing the top and bottom elements from $P$. Then a morphism of lattices $P \to Q$ induces one of posets $\ol{P}\to \ol{Q}$. Next, \cite[Theorem 5.1.5]{wachs} describes equivariant isomorphisms
\[
\wt\Ho_r(\ol{P\times Q}) \cong \bigoplus_{i=0}^r \wt\Ho_i(\ol{P})\otimes \wt\Ho_{r-i-2}(\ol{Q})
\]
which are easily seen to be functorial in $P$ and $Q$.
This implies that after shifting, the poset homology $P\mapsto \wt\Ho_{\bullet-2}(\ol{P})$ is a monoidal functor from bounded posets to graded $\K$-modules. In particular, the homology of a GTCA of lattices itself forms a GTCA.

On Dowling lattices, this construction specializes to give a GTCA structure on the top homology
\[
\wt\Ho_{n-2}(\overline{\D[n](G)})\otimes \wt\Ho_{m-2}(\overline{\D[m](G)}) \to \wt\Ho_{n+m-2}(\overline{\D[n+m](G)})
\]
and all other reduced homology groups vanish. The resulting GTCA will be denoted by $\wt\Ho_{|\bullet|-2}(\ol{\D[\bullet](G)})$
and it will play a role in our topological analysis of orbit configuration spaces.

Note that for $G=1$ the isomorphism $\D(1)\cong\Q_{n+1}$ defines a product structure on the top homology of ordinary partition lattices (shifted by 1).
Here, one considers the $\sym[n]$ action on $\Q_{n+1}$ and can show that  $\wt\Ho_{|\bullet|-2}(\ol{\Q_{\bullet+1}})$ is the TCA of regular representations.
As for a general group $G$, while the sequence of $\G$-representations $\wt\Ho_{n-2}(\ol{\D(G)})$ has been studied in the context of species and generating functions, the algebra structure on this sequence remains virtually unexplored.

\begin{question}
Describe the product structure of the GTCA $\wt\Ho_{|\bullet|-2}(\ol{\D[\bullet](G)})$ and give a combinatorial description of its generators and relations.
\end{question}

\subsection{The GTCA of order relations}

A central tool in the study of poset homology, as well as the cohomology of subspace arrangements, is the \emph{Whitney homology}: we shall think of this as the collection of poset homology groups $\wt\Ho_{\bullet-2}(\ol{P^{\leq p}})$ as $p$ ranges over the elements of $P$. For example, for a finite bounded poset $P=[\hat{0},\hat{1}]$ ranked by $\rk:P\to \NN$,  there is a formal equality
\[
\wt\Ho_{\bullet-2}(\ol{P}) = \pm\sum_{\hat{0} \leq p < \hat{1} } (-1)^{\rk(p)} \wt\Ho_{\bullet-2}(\overline{P^{\leq p}})
\]
in some appropriate Grothendieck group of modules, therefore allowing an inductive description of the homology of $\ol{P}$ (see \cite{sundaram,wachs} for examples). One is thus led to studying the collection of subposets $P^{\leq p}$ for all $p\in P$ -- systematized in this subsection. 

Let $(P,\leq)$ be a poset, and view the order relation $\leq$ as a set of pairs $\{(p_1,p_2)\mid p_1\leq p_2\}\subseteq P\times P.$ The second projection
$$
\pi_2 : {\leq} \to P
$$
has the property that $\pi_2^{-1}(p) \cong P^{\leq p}$, i.e. one realizes the lower intervals in $P$ as the fibers of $\pi_2$. Furthermore, the set $\leq$ is itself naturally ordered by restricting the product order on $P\times P$ (note that this does not coincide with the standard order relation on the poset of intervals, with its order induced from $P^{\operatorname{op}}\times P$). Under our chosen ordering, the projection $\pi_2$ is order preserving, and the isomorphism $\pi_2^{-1}(p) \cong P^{\leq p}$ is one of posets.

Lastly, the passage from a poset $P$ to its order relation $\leq$ is functorial and monoidal with respect to the Cartesian product. Therefore, given a GTCA of posets $P_\bullet$, one gets another by considering $\leq_\bullet$. The projection
\[
\pi_2: {\leq}_\bullet \to P_\bullet
\]
is clearly a morphism of poset GTCAs.

Using this language, the order relation of Dowling posets has a product decomposition as a GTCA compatible with that in Lemma \ref{lem:poset_factorization} above. 
Just as before, this requires forgetting the ordering on $\D[\bullet]$, but now we must also forget the order coming from the second coordinate of $\leq_\bullet.$ 
\begin{definition}
The \textit{right-discretized order relation} of a poset $P$ is the poset $\leq^{r\delta}$ whose underlying set is
\[\leq^{r\delta}=\{(p_1,p_2)\in P\times P\mid p_1\leq p_2\}\]
and whose partial order is given by
\[(p_1,p_2)\leq (q_1,q_2) \iff 
p_1\leq q_1 \text{ and } p_2=q_2.\]
Note that there is an isomorphism $\leq^{r\delta} \cong \coprod_{p\in P} P^{\leq p}$ which is natural in the poset $P$.
\end{definition}

\begin{lemma}[\textbf{GTCA factorization of intervals}]\label{lem:interval_factorization}
Let $G$ be a finite group, and let $\F$ and $T$ be two finite $G$-sets. Pick orbit representatives $s\in [s]$ for $[s]\in \F/G$ and let $G_s$ be the corresponding stabilizer. 
Note that the posets $\D[\bullet]^T(G_s,s)$ form a $G_s$TCA.

Then the associated poset GTCA $\leq_\bullet^{r\delta}$ decomposes as a product of poset GTCAs in a way compatible with the factorization of the set GTCA $\D[\bullet]$ from Lemma \ref{lem:poset_factorization}: 
\begin{equation}
\label{eq:intervalfactors}
\begin{tikzpicture}[baseline=(current  bounding  box.center)]
\node at (-1.6,2.5) {$\leq^{r\delta}$};
\node at (-1.6,0) {$\D[]^T(G,S)$};
\node at (-0.5,2.5) {$\cong$};
\node at (1.1,2.4) {$\displaystyle\prod_{n=1}^\infty\Ind_{G\times\sym}^{\FB}\Q_n$};
\node at (2.7,2.5) {$\times$};
\node at (4.9,2.3) {$\displaystyle\prod_{[s]\in S/G}\Ind_{\FB[{G_s}]}^{\FB}\D[\bullet]^T(G_s,s)$};
\node at (-0.5,0) {$\cong$};
\node at (1.1,0) {$\displaystyle\prod_{n=1}^\infty\Ind_{G\times\sym}^{\FB}(*)$}; 
\node at (2.7,0) {$\times$};
\node at (4.7,-0.1) {$\displaystyle\prod_{[s]\in S/G}\Ind_{\FB[G_s]}^{\FB}(*)_\bullet^{T}(s)$};
\draw[->] (-1.7,2) to node[left] {$\pi_2$} (-1.7,0.5);
\foreach \x in {2.07,5.6}
\draw[->] (\x,1.9) to (\x,0.7);
\end{tikzpicture}
\end{equation}
where $(*)_\bullet^T(s)$ denotes $(*)_\bullet$ when $s\in T$ and $(\check{*})_\bullet$ when $s\notin T$.
In the top row, we let $G\times \sym$ act on the partition lattice $\Q_n$ via its ordinary $\sym$ action with $G$ acting trivially. 
\end{lemma}
\begin{proof}
Fibers of the projection $\pi_2$ are just lower intervals in a Dowling poset, and so we use the interval factorization from \cite[Theorem~A]{BG}, which states that for any $\parts\in\D^T(G,S)$, 
\begin{equation}\label{eq:intervals}
\D[A]^T(G,\F)^{\leq \parts} \cong 
\prod_{B\in\beta} \Q_B 
\times
\prod_{[s]\in \F/G} \D[{z^{-1}([s])}]^T(G_s,s).
\end{equation}

In particular, when $[1,\dots,n]\in\D[{[n]}](G,S)$ is a single block with trivial $G$-coloring, we have an isomorphism $\Pi_n\cong\D[{[n]}](G,S)^{\leq[1,\dots,n]}$, which then induces a map from the GTCA generated by $\Pi_n$ to $\leq^{r\delta}$.
When $(\emptyset,s)\in\D[{[n]}](G,S)$ is an element whose zero block is $[n]$ and colored by $s$, we have an isomorphism $\D[{[n]}]^T(G_s,s)\cong \D[{[n]}]^T(G,S)^{\leq(\emptyset,s)}$, which then induces a map from the GTCA generated by $\D[\bullet]^T(G_s,s)$ to $\leq^{r\delta}$.
Lemma \ref{lem:poset_factorization} describes how to write an element $\parts\in\D[A]^T(G,S)$ as a product of its blocks, and \eqref{eq:intervals} describes how to lift this so that $\D[A]^T(G,S)^{\leq\parts}$ can be written as an element in the product.

Because \eqref{eq:intervals} is an isomorphism of posets, the product decomposition of the right-discretized order relation $\leq_{\bullet}^{r\delta}$ holds as a GTCA of posets, rather than just as sets.
\end{proof}

\begin{example}\label{ex:interval_factors}
Consider a finite group $G$ acting trivially on $T=S=\{+,-\}$.
The interval underneath a single-block partition $\{1,2\}$ in $\D[{[2]}]^\pm(G,\pm)$ with any $G$-coloring is isomorphic to $\Q_{[2]}$, a poset with two elements, independently of the chosen $G$-coloring.
Similarly, in $\D[\{5\}]^{\pm}(G,\pm)$, the interval underneath a maximal element, say the zero block $\{5\}$ colored by $+$, is isomorphic to $\D[\{5\}]^{\pm}(G,+)$, which is again a poset with two elements.

Now let us consider the product $12\cdot 34\cdot 5_+$ in $\D[{[5]}]^{\pm}(G,\pm)$: the interval underneath should be a product of the previously mentioned intervals. We depict this isomorphism in Figure \ref{fig:interval_factors}. Note that this diagram sits inside of the much larger poset 
$\D[{[5]}]^{\pm}(G,\pm)$.

\begin{figure}[hb]
\begin{tikzpicture}
\node at (0,-1) {{\tiny $\left(\D[\{1,2\}]^{\pm}(G,\pm)\right)^{\leq 12}\times
\left(\D[\{3,4\}]^{\pm}(G,\pm)\right)^{\leq 34}
\times\left(\D[\{5\}]^{\pm}(G,\pm)\right)^{\leq 5_+}
\cong$}{\small$\left(\D[{[5]}]^{\pm}(G,\pm)\right)^{\leq 12\cdot 34 \cdot 5_+}$}};
\draw[blue,thick,-] (-4.5,0)--(-4.75,1);
\draw[blue,thick,-] (3,1)--(4,0);
\draw[cyan,thick,-] (-2,0)--(-2,1);
\draw[cyan,thick,-] (4,1)--(4,0);
\draw[red,thick,-] (0.5,0)--(0.75,1);
\draw[red,thick,-] (5,1)--(4,0);
\foreach \x/\y in {-4.5/-4.75,-2/-2,0.5/0.75} 
{\draw[fill=black] (\x,0) circle (0.5ex);
\draw[fill=black] (\y,1) circle (0.5ex);};
\node at (-3.1,0.5) {$\times$};
\node at (-0.45,0.5) {$\times$};
\node at (2.05,0.5) {$\cong$};
\foreach \x in {3,4,5} {\draw[-] (\x,2)--(4,3);};
\draw[-] (3.4,1.6)--(3,2)--(3,1)--(4,2)--(5,1)--(5,2)--(4.6,1.6);
\draw[-] (3.6,1.4)--(4,1)--(4.4,1.4);
\node at (2.5,2) {\tiny $12\cdot 34$};
\node at (5.5,2) {\tiny $34\cdot 5_+$};
\node at (4,3.2) {\tiny $12\cdot 34\cdot 5_+$};
\node at (-5,1.2) {\tiny $12$};
\node at (-2,1.3) {\tiny $34$};
\node at (1,1.2) {\tiny $5_+$};
\end{tikzpicture}
\caption{Factorization of an interval inside $\D[{[5]}]^{\pm}(G,\pm)$. See Example \ref{ex:interval_factors}.}
\label{fig:interval_factors}
\end{figure}
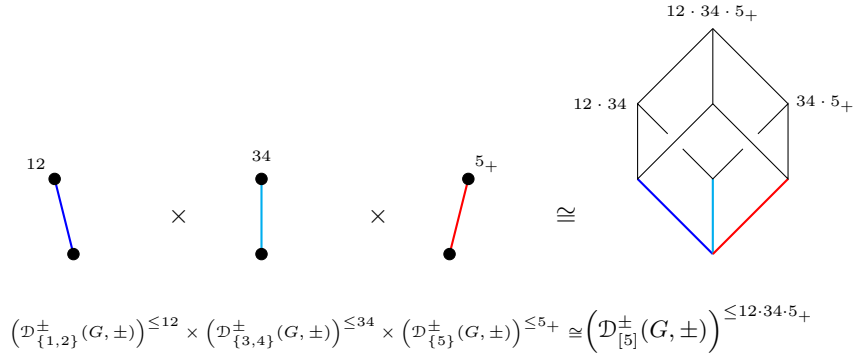
\end{example}

\subsection{Whitney homology}

The interval factorization of Lemma \ref{lem:interval_factorization} immediately translates to a factorization of Whitney homology, defined as follows.

\begin{definition}\label{def:whitney}
The \textit{Whitney homology} of a finite poset $P$ containing a bottom element is
\[\WH_*(P):=\bigoplus_{p\in P} \wt{\Ho}_{*-2}(\ol{P^{\leq p}}).\]
\end{definition}

Just as with poset homology, $P\mapsto\WH_*(P)$ is a monoidal functor from finite bounded-below posets to graded $\K$-modules (where $\K$ is as in Conventions \ref{conventions}). In particular, when $P_\bullet$ is a GTCA of finite bounded-below posets, $\WH_*(P_\bullet)$ is a GTCA of graded $\K$-modules.

\begin{theorem}[\textbf{Whitney homology factorization}]\label{thm:Whitney_factorization}
Let $G$ be a finite group, and let $S$ and $T$ be two finite $G$-sets. Pick orbit representatives $s\in [s]$ for orbits $[s]\in \F/G$ and let $G_s$ denote the respective stabilizers.
The GTCA of Whitney homology $\WH_{*}(\D[\bullet]^T(G,S))$
factors as a GTCA of graded $\K$-modules into the product
\begin{equation} \label{eq:Whitney_factorization}
\bigotimes_{n=1}^\infty \Ind_{G\times \sym}^{\FB} \left(\wt\Ho_{n-3}(\ol{\Q_n})\right) \otimes \bigotimes_{[s]\in \F/G} \Ind_{\FB[G_s]}^{\FB} \left( \wt\Ho_{|\bullet|-2}(\ol{\D[\bullet]^T(G_s, s)})\right)
\end{equation}
where the graded $\K$-module $\wt\Ho_{n-3}(\ol{\Q_n})$ is homogeneous of degree $n-1$, and the degree of $\wt\Ho_{|\bullet|-2}(\ol{\D[\bullet]^T(G_s)})$ is $|\bullet|$.
\end{theorem}
\begin{proof}
As stated above, the right-discretized order relation of a poset $P$ satisfies $\leq^{r\delta}\cong \coprod_{p\in P}P^{\leq p}$. 
This isomorphism is in fact monoidal in $P$: for posets $P_1$ and $P_2$ it is clear that there are natural poset isomorphisms
\[
\leq_1^{r\delta}\times \leq_2^{r\delta} \cong \coprod_{(p_1,p_2)} P_1^{\leq p_1}\times P_2^{\leq p_2} \cong \coprod_{(p_1,p_2)} (P_1\times P_2)^{\leq (p_1,p_2)}\cong (\leq_1\times \leq_2)^{r\delta}.
\]
We thus notice that the Whitney homology functor is just the composition of the functors
\[
\WH_*(P): P\mapsto \leq^{r\delta} \mapsto \wt\Ho_{*-2}(\ol{\leq^{r\delta}}) 
\]
where we extend the operation $Q\mapsto \ol{Q}$ to disjoint unions of lattices by removing the top and bottom element of each connected component.

We claim that each of these operations commute with products and induction. Indeed, both the product of GTCAs and induction are built from finite disjoint unions and products. Therefore, every functor that is both additive and multiplicative, i.e. monoidal with respect to $\coprod$ and $\otimes$ will automatically preserve products and inductions of GTCAs. In the case at hand, the functor $P \mapsto \leq^{r\delta}$ is clearly additive and multiplicative. The shifted homology functor was discussed in \S\ref{sec:poset_homology}, where its multiplicative property was observed and additivity is obvious.

It follows immediately from Lemma \ref{lem:interval_factorization} that the Whitney homology factors into a product of GTCAs, induced from $\wt\Ho_{*-2}(\ol{\Q_n})$ and $\wt\Ho_{*-2}(\ol{\D[\bullet]^T(G_s,s)})$. These posets are geometric lattices and thus have homology only in dimension $\rk - 2$, where $\rk(\Q_n)=n-1$ and $\rk(\D^T(G_s,s))=n$. 
\end{proof}

We shall see in the following section how the Dowling posets, and in particular their Whitney homology, are a key ingredient in understanding the (co)homology of orbit configuration spaces.
Before we proceed to discuss general orbit configuration spaces, there are special cases in which the Whitney homology already coincides with the cohomology. 
\begin{corollary}\label{cor:affine}
Let $X=\mathbb{A}^d$ be an affine space over either $\R$ or an algebraically closed field.
Let a finite group $G$ act almost freely on $X$ by affine transformations. Then for any finite $G$-invariant set $T\subset X$, the cohomology of $\Conf^\bullet_G(X,T)$ decomposes into the product of induced GTCAs given in \eqref{eq:Whitney_factorization} of Theorem \ref{thm:Whitney_factorization}.
\end{corollary}
\begin{proof}
The claim follows from Goresky-MacPherson's formula \cite[Part III]{GM}: it shows that the complement for any arrangement of affine subspaces has cohomology given by the Whitney homology of the corresponding intersection poset, possibly up to simple dimension shifts. 
But in \cite[Theorem~C]{BG} we show that these intersection posets are exactly the corresponding Dowling posets.
\end{proof}

\begin{example} \label{ex:exact_product_formula}
Special cases to which Corollary \ref{cor:affine} applies include:
\begin{itemize}
\item The classical configuration space $\Conf^\bullet(\mathbb{R}^d)$, when $G=1$ and $T=\emptyset$.
\item The configuration space of a punctured space $\Conf^\bullet(\mathbb{R}^d\setminus \{r_1,\ldots,r_k\})$, when $G=1$ and $T=\{r_1,\dots,r_k\}$.
\item The complement of the type B/C and D root systems and their complexifications, when $G=\Z_2$ and either $T=\{0\}$ or $T=\emptyset$, respectively.
\item The complement of a Dowling arrangement $\Conf_{\mu_k}^\bullet(\C\setminus\{0\})$, when the $k$-th roots of unity act by scalar multiplication $\mu_k\curvearrowright \mathbb{A}$, and $T=\{0\}$. 
\end{itemize}
\end{example}

\subsection{Combinatorics of orbit configuration spaces}
\label{sec:stratification}

Let $X$ be a separated $G$-space (as in Conventions \ref{conventions}) with an almost free action, i.e. where the set of singular points for the $G$-action
\[
\Sing_G(X) := \bigcup_{1\neq g\in G} X^g
\]
is finite, denoting by $X^g$ the set of points fixed by $g$. Fix a finite $G$-invariant subset $T\subseteq X$, and let
\[S := \Sing_G(X)\cup T.\] 

Every element $\beta=(\{\wt{B}_1,\dots,\wt{B}_\ell\},z)\in\D^T(G,S)$ defines a subspace $X^\beta\subseteq X^n$ as follows.
The partial $G$-partition of $[n]$ specifies which coordinates are related by an application of $g\in G$ and the zero block specifies which coordinates land on an element $s\in\Sing_G(X)\cup T$.
More explicitly,
\begin{equation}\label{eq:Xbeta}
X^{\beta} := X^{\wt{B}_1}\times\cdots\times X^{\wt{B}_\ell}\times X^z\subseteq X^{B_1}\times\cdots\times X^{B_\ell}\times X^Z\cong X^n, 
\end{equation}
where for a block $B$ with $G$-coloring $b:B\to G$, the subspace $X^{\wt{B}}\subseteq X^B$ consists of functions $x_B:B\to X$ satisfying
\[b(i)^{-1}.x_B(i) = b(j)^{-1}.x_B(j) \quad \forall i,j\in B,\]
and $X^z$ is the single element of $X^Z$ defined by $z:Z\to S\subseteq X$.

The collection of $X^\beta$ as $\beta$ runs over $\D=\D^T(G,S)$ describes the closed strata in a stratification of $X^n$. The locally-closed strata $U^\beta$ whose disjoint union is $X^n$ are given by $U^\beta = X^\beta\setminus\cup_{\alpha\in\D^{\geq\beta}} X^\alpha$, and the open stratum corresponding to the trivial partition with empty zero block (the minimum of $\D$) is $\Conf_G^n(X,T)$. 
We refer the reader to \cite[Theorem~C]{BG} for more details on this stratification; although that theorem assumed that $X$ is connected, the stratification does not require this assumption.

 Consider the `incidence variety' over the Dowling poset:
\begin{center}
\begin{tikzpicture}
\node at (-1.5,2.5) {$\mathfrak{L}_n = \{ (x, \beta)\in X^n\times \D \mid x\in X^{\beta} \}$};
\node at (-3.5,0) {$\D^T(G,S)$};
\node at (2,2.5) {$\ni$};
\node at (2,0) {$\ni$};
\node at (3.2,2.5) {$(x,\beta)$};
\node at (3.2,0) {${\beta}$};
\draw[->] (-3.5,2) to node[left] {$\pi_2$} (-3.5,0.5);
\draw[|->] (3.2,1.9) to (3.2,0.7);
\end{tikzpicture}
\end{center}
The fiber over an element $\beta\in \D^T(G,S)$ is simply the topological space $X^{\beta}$. In this sense the incidence variety is the tautological bundle over $\D^T(G,S)$.

Letting the parameter $n$ vary, one sees that the collection of incidence varieties is in fact a GTCA of topological spaces, and the projection $\pi_2$ is a map of GTCAs.
A geometric version of Lemma \ref{lem:interval_factorization} is the following,

\begin{lemma}[\textbf{GTCA factorization of incidence varieties}] \label{thm:topology_factorization}
Let $X$ be an almost free $G$-space (as in Conventions \ref{conventions}), let $T\subset X$ be a finite $G$-invariant subset, and let $S=\Sing_G(X)\cup T$.
The GTCA of incidence varieties factors as a product of topological GTCAs, compatibly with the factorization of $\D[\bullet]^T(G,S)$:
\begin{center}
\begin{tikzpicture}
\node at (-1.7,2.5) {$\mathfrak{L}$};
\node at (-1.7,0) {$\D[]^T(G,S)$};
\node at (-0.5,2.5) {$\cong$};
\node at (1.1,2.4) {$\displaystyle\prod_{n=1}^\infty\Ind_{G\times\sym}^{\FB}X$};
\node at (2.7,2.5) {$\times$};
\node at (4.9,2.3) {$\displaystyle\prod_{[s]\in \F/G}\Ind_{\FB[{G_s}]}^{\FB} (*)^T_\bullet(s)$};
\node at (-0.5,0) {$\cong$};
\node at (1.1,0) {$\displaystyle\prod_{n=1}^\infty\Ind_{G\times\sym}^{\FB}(*)$}; 
\node at (2.7,0) {$\times$};
\node at (4.9,-0.1) {$\displaystyle\prod_{[s]\in \F/G}\Ind_{\FB[G_s]}^{\FB}(*)_\bullet^{T}(s)$};
\draw[->] (-1.7,2) to node[left] {$\pi_2$} (-1.7,0.5);
\foreach \x in {2.1,5.6}
\draw[->] (\x,1.9) to (\x,0.7);
\end{tikzpicture}
\end{center}
where $G\times \sym$ acts on $X$ via the provided action of $G$ and, as in Lemma \ref{lem:poset_factorization}, $(*)^T_\bullet(s)$ is either a point or the empty space.
\end{lemma}
\begin{proof}
Lemma \ref{lem:poset_factorization} gives a factorization of $\D[\bullet]^T$. The generator $*$ of the factor indexed by $n$ corresponds to the diagonal $\Delta\subseteq X^n$, which is $(G\times \sym)$-equivariantly isomorphic to $X$. As for the factor labeled by an orbit $[s]$, the point in $(*)^T_n(s)$ corresponds to the point $\{(s,\ldots, s)\}\subseteq X^n$, which can be equivariantly identified with $(*)^T_n(s)$ itself.
The multiplication of these subspaces is given by the Cartesian product of topological spaces, and so the fiber over any $\beta\in\D$ is the product of the diagonals and points as prescribed by $\beta$ and described in \eqref{eq:Xbeta}.
\end{proof}

\subsection{The collision spectral sequence} \label{sec:spectral_sequence}

Now we have two compatible product decompositions: an order theoretic one in Lemma \ref{lem:interval_factorization} and a topological one in Lemma \ref{thm:topology_factorization}. After applying homology, we combine them to get a handle on the Borel--Moore homology of $\Conf_G^\bullet(X,T)$ (with coefficients as in Convention \ref{conventions}). 

The proposed factorization does not exist universally for all spaces (although it applies to a large class of them, see Remark \ref{rmk:differentials} and Corollary \ref{cor:i-acyclic} below). Rather, the factorization appears at a finite stage of a spectral sequence, associated with a natural filtration we now discuss.

We define a filtration on the homology $\Ho^{BM}_*(\Conf_G^n(X,T))$, natural with respect to proper maps. The idea here is that $F_k\Ho^{BM}_*$ are all the Borel-Moore homology classes that are restricted from partial configuration spaces in which only $k$ simultaneous collisions are excluded.

Recall from \S\ref{sec:stratification} that the Dowling poset $\D(X,T)$ indexes a stratification of $X^n$, whose closed strata are denoted by $X^\beta$ for $\beta\in\D(X,T)$ as defined in \eqref{eq:Xbeta}.
This Dowling stratification of $X^n$ gives rise to a diagram of open sets in $X^n$, where $\beta\in \D(X,T)$ corresponds to a partial configuration space: the open set 
\[
\Conf^{\beta}(X) = X^n \setminus \bigcup_{\zero < \alpha \leq \beta}X^\alpha
\]
where one only removes the subspaces $x_i= g.x_j$ and $x_i=t$ that contain $X^\beta$. Then $\beta < \beta'$ implies $\Conf^\beta(X) \supset \Conf^{\beta'}(X)$.

Since Borel-Moore homology is contravariant with respect to open inclusions, the containments $\Conf^n_G(X,T)\subseteq \Conf^\beta(X)$ induce a diagram of subspaces 
\[
 \Ho^{BM}_*(\Conf^n_G(X,T)) ^{\beta} := \operatorname{Im}\left(\Ho^{BM}_*(\Conf^\beta(X)) \to \Ho^{BM}_*(\Conf^n_G(X,T))\right)
\]
for which $\beta < \beta'$ now implies $\Ho^{BM}_*(\Conf^n_G(X,T)) ^{\beta} \subseteq \Ho^{BM}_*(\Conf^n_G(X,T)) ^{\beta'}$. Then the rank function on $\D(X,T)$ gives rise to a filtration as follows.

\begin{definition}[\textbf{Collision filtration}]
    With the notation of the previous paragraph, set
    \[
    F_k(\Ho^{BM}_*(\Conf^n_G(X,T))) := \sum_{\rk(\beta)\leq k} \Ho^{BM}_*(\Conf^n_G(X,T)) ^{\beta}.
    \]
    which is already defined at the level of Borel-Moore chains $F_k(C_*^{BM}(\Conf^n_G(X,T)))$, and more generally with coefficients in any sheaf on $X$.
\end{definition}

Clearly the filtration is natural in all inputs $(G,X,T)$ in the obvious way. More importantly, the filtration is compatible with the GTCA structure: the $\G$-action on $X^n$ permutes the sets $\Conf^{\beta}(X)$ while preserving $\rk(\beta)$, thus the induced action on homology respects the filtration. Furthermore, since for every $\beta\in \D[n]$ and $\beta'\in \D[m]$ there is a natural open inclusion $\Conf^{\beta\times \beta'} \subset \Conf^\beta \times \Conf^{\beta} \subset X^{n+m}$, it follows that there is a multiplication map
\[
\Ho^{BM}_*(\Conf^n_G(X,T)) ^{\beta} \otimes \Ho^{BM}_*(\Conf^m_G(X,T)) ^{\beta'} \to \Ho^{BM}_*(\Conf^{n+m}_G(X,T)) ^{\beta\times \beta'}
\]
and in particular $F_k \otimes F_\ell \to F_{k+\ell}$. We shall therefore use the GTCA notation \[F_k(\Ho^{BM}_*(\Conf^\bullet_G(X,T)))\] to refer to the respective filtrations on all powers of $X$ simultaneously. 
This is \emph{the collision filtration}.

The spectral sequence associated with this filtration admits a natural product decomposition of GTCAs, as the following shows.

\begin{theorem}[\textbf{Spectral sequence factorization}] \label{thm:ss_factorization}
Let $X$ be an almost free $G$-space as in Conventions \ref{conventions}, let $T\subseteq X$ be a finite $G$-invariant subset, and let $S=\Sing_G(X)\cup T$. Pick orbit representatives $s\in [s]$ for every $[s]\in \F/G$ and let $G_s\leq G$ denote its stabilizer.
The collision filtration gives rise to a spectral sequence of GTCAs converging to the GTCA $\Ho_{*}^{BM}(\Conf^{\bullet}_G(X,T))$, with $E^1\cong$
\begin{equation}\label{eq:ss_product_factorization}
\bigotimes_{n=1}^{\infty}\Ind_{G\times \sym}^{\FB} \left(\Ho^{BM}_*(X) \boxtimes \wt{\Ho}_{n-3}(\ol{\Q}_n)\right)
\otimes \bigotimes_{[s]\in \F/G} \Ind_{\FB[G_s]}^{\FB} \wt{\Ho}_{|\bullet|-2}\left(\ol{\D[\bullet]^T(G_s,s)}\right),
\end{equation}
where $\Ho^{BM}_*(X) \boxtimes \wt{\Ho}_{n-3}(\ol{\Q}_n)$ is placed in bidegree $(n-1,*)$ and $\wt{\Ho}_{|\bullet|-2}(\ol{\D[\bullet]^T(G_s,s)})$ is in bidegree $(\bullet,0)$. These GTCAs are in the category of bigraded modules with standard tensor product and (graded-commutative) symmetry.

Moreover, the spectral sequence and the product factorization are natural with respect to proper $G$-equivariant maps.
\end{theorem}

\begin{proof}
In \cite[Lemma 4.12]{petersen} Petersen constructs a spectral sequence of TCAs converging to $\Ho^{BM}_*(\Conf^\bullet_G(X,T))$, which by  \cite[Theorem~D]{BG} has $E^1$ given by
\begin{equation*}
E^1_{p,q}[\bullet] = \bigoplus_{\substack{\beta\in \D[\bullet] \\ \rk(\beta) = p} } \Ho^{BM}_q\left(X^\beta\right) \otimes \wt\Ho_{p-2}\left(\ol{\D[\bullet]^{\leq \beta}}\right),
\end{equation*}
where $\D[n]= \D[n]^T(G,S)$. 
Note that while Petersen explicitly considered only the symmetric group action on each term, his arguments apply to the $G$-equivariant context and in fact give a spectral sequence of GTCAs.

We use Theorem \ref{thm:Whitney_factorization} and Lemma \ref{thm:topology_factorization} (applying Borel--Moore homology to the second), in which we have factorizations of two GTCAs of graded $\K$-modules. By taking their pointwise tensor product over the GTCA $\D[\bullet]^T(G,S)$, one obtains the factorization stated here. The details are completely straightforward and thus omitted.

It remains to show that Petersen's construction in fact coincides with the spectral sequence associated with our collision filtration.
As a quick reminder, let $j:\Conf^n_G(X,T)\into X^n$ denote the inclusion. Starting with a sheaf $\mathcal{F}$ on $X^n$, Petersen resolves the sheaf $j_! j^{-1} \mathcal{F}$ which computes $\Ho^*_c(\Conf^n_G(X,T);\mathcal{F})$ by the complex
\[
\mathcal{F} \to \bigoplus_{0<\beta\in \D} \mathcal{F}|_{X^\beta} \to \bigoplus_{0<\beta_1<\beta_2} \mathcal{F}|_{X^{\beta_1}} \to \ldots \to \bigoplus_{0<\beta_1<\ldots< \beta_n} \mathcal{F}|_{X^{\beta_n}} \to \ldots
\]
and filters it by $\rk(\beta_n)$. This filtration produces the aforementioned spectral sequence. Applying the functor computing global sections of the Verdier dual $R\Gamma \circ \mathbb{D}$ to this complex gives a filtered chain complex computing $\Ho^{BM}_*(\Conf^n_G(X,T);\mathcal{F})$.

Similarly, for a partial configuration space $j^\beta: \Conf^\beta(X) \into X^n$, the sheaf $(j^\beta)_! (j^\beta)^{-1}\mathcal{F}$ is resolved in the same way, but with the poset $\D$ replaced by the subposet $\D^{\leq \beta}$. Thus the complex computing $\Ho^{BM}_*(\Conf^\beta(X))$ naturally sits as a subcomplex of the one computing $\Ho^{BM}_*(\Conf^n_G(X,T))$, and this inclusion realizes the restriction from the former to the latter.

Since the subcomplex of Petersen's filtration degree $p$ is precisely the sum over all terms in $\D^{\leq \beta}$ for $\rk(\beta)\leq p$, it is precisely the chains restricted from those $\Conf^\beta$. This is our definition of the collision filtration.
\end{proof}

\begin{remark}
Dan Peterson and Phil Tosteson communicated to us that they had each observed this factorization in the special case that $G$ is trivial and $S$ is empty: in the language of species, one views this as a composition with an exponential.
For more general $G$ and $S$, we have more elaborate group inductions as well as additional factors corresponding to non-free orbits.
\end{remark}

\begin{remark}[\textbf{Factorization at the space level}]
One could build a more systematic framework to show a factorization at the space level, making sense of the object 
\[\bigsqcup_{\beta\in\D} X^\beta\times\D^{\leq\beta}\]
in a category that allows us to pair together topological spaces and posets. 
While the factorization is clear in the category of sets, one needs the additional structure on the objects in order to apply homology functors.
\end{remark}

\begin{remark}[\textbf{Differentials}]
\label{rmk:differentials}
This result would still be of limited use if nothing could be said about the differentials of the spectral sequence. Fortunately, several general statements could be made. 
\begin{itemize}
\item When $X$ is a smooth projective variety, there can be at most one nonzero differential (by a standard weight argument). This nontrivial differential is completely determined by what it does to the generators of the GTCA.
\item When $X$ is $i$-acyclic (the map $\Ho_*(X)\to \Ho^{BM}_*(X)$ is zero), all differentials must vanish \cite{Arabia,petersen-formal}. For example, a product of any space by affine spaces and copies of the multiplicative group $\mathbb{G}_m$ is $i$-acyclic.
In this case, it is not known whether the Borel-Moore homology GTCA factors as a product of inductions, even though the GTCAs are isomorphic as species (forgetting the algebra structure).
\end{itemize}
Other general statements involve the representation stability, which is the subject of the next section. In short, the kernel and cokernel of every differential are finitely generated modules for certain algebras actions, and this constrains the $\G$-representation that may occur.
\end{remark}

\section{Representation stability and secondary stability}\label{sec:repstability}

The product formula of Theorem \ref{thm:ss_factorization} gives a good handle on how to generate the $E^1$-page computing the homology of configuration spaces as a GTCA. Unfortunately, the differentials are difficult to describe in terms of these generators. The main goal of this section is to address this difficulty and shed light on the structure of the homology $\Ho^{BM}_*(\Conf^n_G(X,T))$ using the tools of representation stability.

Stated vaguely, this section will show that under mild hypotheses on $X$, the Borel-Moore homology groups $\Ho^{BM}_*(\Conf_G^n(X,T))$ stabilize as sequences of representations of the various wreath products. One of the central contributions of the representation-stability point of view is that this notion of stability is best understood as the finite-generation of modules over GTCAs. Then the Noetherian property of certain GTCAs ensures that stability is a robust property, preserved under subquotients and extensions. It is therefore enough to establish finite-generation at the $E^1$-page of a spectral sequence converging to the modules in question.

\subsection{Geometric criteria for finite-generation} \label{subsec:generation_locus}
Here we present a simple geometric technique for identifying many structures of finitely-generated modules on a bigraded GTCA $E_{*,*}[\bullet]$ as occurs in the $E^1$-page of the spectral sequence computing $\Ho^{BM}_*(\Conf_G^\bullet(X,T))$. Very briefly, it says that there exists a polygon whose corners govern the finitely-generated module structures on $E$. 

Comparing GTCA generators appearing at different $\bullet=n$ gets rather confusing. An effective trick to simplify this problem is to work as though all generators appear already in $\bullet=1$, which is formally similar to adjoining an $n$-th root to generators $\in E[n]$. Making this approach systematic, we make the following definition.

\begin{definition}[\textbf{The generation locus}] \label{def:genlocus}
Let $E_{*,*}[\bullet]$ be a bigraded GTCA, generated by the finite-dimensional subspaces $V_i \leq E_{p_i,q_i}[n_i]$ with $i\in \NN$.

For every $e\in E_{p,q}[n]$ write $v(e):= \left( \frac{p}{n} , \frac{q}{n} \right) \in \mathbb{Q}^2$ and define the \emph{generation locus} to be
\[
\Gen := \left\{ v(V_i) \mid i\in \NN \right\} \subset \mathbb{Q}^2.
\]
These are the (fractional) bidegrees in which generators would have appeared had they all existed already when $\bullet=1$.
\end{definition}

Consider a point $v\in \mathbb{Q}^2$. Let $A_v\leq E$ be the subalgebra generated by all subspaces $V_i$ with $v(V_i) = v$, and note that multiplication by $A_v$ preserves the subspaces 
\begin{equation} \label{eq:submod}
E_{(p,q)+\bullet v} [n+\bullet]
\end{equation}
where we declare that $E_{a,b}[c]=0$ unless $a,b,c\in \ZZ_{\geq 0}$. Thus $E$ breaks up into a collection of $A_v$-submodules.

\begin{lemma}[\textbf{Corner criterion for finite generation}]\label{lem:corner_fin_gen}
Let $E_{*,*}$ be a bigraded GTCA, generated by subspaces $V_i\leq E_{p_i,q_i}[n_i]$.
Suppose that for every point $v\in \Gen$ and fixed $n\in \NN$ the corresponding generating subspace 
\[\bigoplus_{\substack{v(V_i)=v \\ n_i = n}} V_i\]
is finite dimensional.
Then for every isolated corner $v_0$ of the closed convex hull $\ol{\conv(\Gen)}$, the $A_{v_0}$-submodules $E_{(p,q)+\bullet v_0}[n+\bullet]$ that make up $E_{*,*}$ are finitely-generated.

Quantitatively, if $L\subset \mathbb{R}^2$ is a line meeting $\ol{\conv(\Gen)}$ only at a corner $v_0\in \Gen$, and there are no other points in $\Gen$ within distance $\epsilon>0$ from $L$, then $E_{(p,q)+\bullet v_0}[n+\bullet]$ is generated as an $A_{v_0}$-module by the finitely many products $V_{i_1}\cdot \ldots \cdot V_{i_k}$ with $\sum n_{i_j} \leq \frac{\operatorname{dist}(nL,(p,q))}{\epsilon},$ and each for these is finite-dimensional.
\end{lemma}

\begin{proof}
Let $L(x,y) = ax+by + c = 0$ be a defining equation for a line $L$ meeting $\ol{\conv(\Gen)}$ at the corner $v_0\in \Gen$ and is at least distance $\epsilon>0$ to any other point in $\Gen$. Normalize the equation $L$ so that $|L(v)| = \operatorname{dist}(v,L)$ for every $v\in \mathbb{R}^2$ and takes non-negative values on $\Gen$. In particular, one has $L(v)\geq \epsilon$ for all $v (\neq v_0)\in \Gen$.

Next, with the notation of Definition \ref{def:genlocus} above, define a `height' function $|e| := n\cdot L(v(e)) = ap+bq+nc$ when $e\in E_{p,q}[n]$. Linearity in $(p,q,n)$ implies $|e\cdot e'| = |e|+|e'|$ for every two homogeneous elements of $E$. Furthermore, the hypotheses of the previous paragraph imply
\begin{itemize}
    \item $|V_i| = 0$ if and only if $v(V_i) = v_0$,
    \item and otherwise $|V_i|\geq n_i\epsilon$.
\end{itemize}

Consider the $A_{v_0}$-module $E_{(p,q)+\bullet v_0}[n+\bullet]$: its height is given by the constant $\lambda := nL(\frac{p}{n},\frac{q}{n})$ since $L(v_0)=0$ (note that $\lambda$ is well defined even if $n=0$). By the assumption on $E_{*,*}$, this subspace is generated by various products $V_{i_1}\ldots V_{i_k}$. But since the height is additive and non-negative, products of generators that produce elements of height $\lambda$ must have $|V_i|\leq \lambda$.

We claim that for every fixed $\lambda \in \mathbb{R}$ there are finitely many elements that generate all spaces $V_i$ with $v(V_i)\neq v_0$ and $|V_i|\leq \lambda$. Indeed, the second property of the height above implies $n_i\epsilon\leq \lambda$, so $n_i \leq \lambda/\epsilon$. Furthermore, the triangle bounded by the cone angle of $\ol{\conv(\Gen)}$ at $v_0$ and $L(x,y)\leq \lambda$ is compact (see Figure \ref{fig:triangle}), so meets every lattice $\frac{1}{n}\mathbb{Z}$ only at finitely many points. Since the generating spaces $V_i$ with $|V_i|\leq \lambda$ must correspond to bounded $n_i\leq \lambda/\epsilon$ and have $v(V_i)$ among the corresponding finite sets of lattice points, there is a finite list of elements in $E$ that generates all of them by hypothesis.

It follows that the $A_{v_0}$-module $E_{(p,q)+\bullet v_0}[n+\bullet]$ is generated under multiplication by $A_{v_0}$ by a finite list of elements -- those in products $V_{i_1}\ldots V_{i_k}$ of total height $\lambda$. Again, the lower bound on non-zero height implies that in all such products $\epsilon\sum n_{i_j} \leq \lambda$, bounding the possible degrees as claimed.
\begin{figure}[ht]
\begin{tikzpicture}[scale=2]
\node at (-0.1,1.1) {$v_0$};
\node at (1.3,1.1) {$L(x,y)=\lambda$};
\node at (1,1.55) {$L(x,y)=0$};
\tikzstyle{every node}=[draw,circle,fill,minimum size=2.5pt,inner sep=0pt]
\foreach \x in {0,1,2,3,4,5}
{\node at (0,\x/5) {};};
\foreach \x in {0,1,2,3,4}
{\node at (1/5,\x/5) {};};
\foreach \x in {0,1,2,3}
{\node at (2/5,\x/5) {};};
\foreach \x in {0,1,2}
{\node at (3/5,\x/5) {};};
\foreach \x in {0,1}
{\node at (4/5,\x/5) {};};
\node at (1,0) {};
\draw (-0.5,.5)--(.5,1.5);
\draw (-0.5+.3,.5-0.4)--(.5+.3,1.5-0.4);
\draw[fill=gray, fill opacity=0.3, color=gray] (0,0)--(0,1)--(1,0)--(0,0);
\draw[fill=blue, fill opacity=0.3, color=blue]
(0,1)--(0,.3)--(1.3/2-.3,1.3/2)--(0,1);
\end{tikzpicture}
\caption{Triangle bounded by cone angle of $v_0$ and $L(x,y)\leq \lambda$}
\label{fig:triangle}
\end{figure}
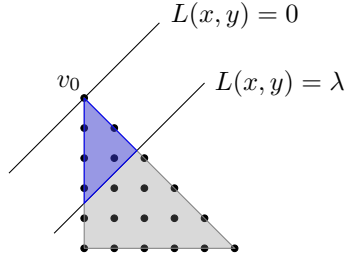
\end{proof}

In practice, one is most often interested in diagonals $\displaystyle\bigoplus_{p+q=i}E_{p,q}$ and their finite-generation as modules over GTCAs. The following slight adjustment to the corner criterion argument of Lemma \ref{lem:corner_fin_gen} extends the finite-generation result to the entire diagonal.
\begin{lemma}[\textbf{Slope $-1$ criterion for finite generation of diagonals}]\label{lem:slope_1}
Let\\ $E_{*,*}$ be as in Lemma \ref{lem:corner_fin_gen} -- a bigraded GTCA generated by subspaces $V_i\leq E_{p_i,q_i}[n_i]$ such that for every $v\in \Gen$ and $n\in \NN$ the corresponding generating subspace \[\bigoplus_{\substack{v(V_i)=v\\ n_i=n}} V_i\] is finite dimensional.

If $v_0=(r_0,t_0)$ is a corner of the closed convex hull $\ol{\conv(\Gen)}$ and the line $L$ of slope $(-1)$ through $v_0$ is at least of distance $\epsilon>0$ to any other point in $\Gen$ (see Figure \ref{fig:slope-1} and compare with Lemma \ref{lem:corner_fin_gen}), then for every $i\in\ZZ$ the diagonals
\begin{equation}\label{eq:diagonal}
    \bigoplus_{p+q = i+\bullet(r_0+t_0)}E_{p,q}[n+\bullet]
\end{equation}
form a finitely-generated $A_{v_0}$-module. More explicitly, this module is generated by products $V_{i_1}\cdot \ldots \cdot V_{i_k}$ with $\sum n_{i_j} \leq \frac{\operatorname{dist}(nL,(i,0))}{\epsilon}$.
\end{lemma}
\begin{figure}[ht]
\begin{tikzpicture}[scale=2]
\node at (0.1,2.1) {$v_0$};
\node at (1.3,0.9) {$L$};
\draw[<->] (1.5+.05,.5-.05) -- node[pos=.3,below] {$\epsilon$} (1.3+.05,.3-.05);
\tikzstyle{every node}=[draw,circle,fill,minimum size=2.5pt,inner sep=0pt]
\node at (0,0) {};
\node at (0,1) {};
\node at (0,2) {};
\foreach \x in {0,1,2} {
\foreach \y in {1,2,3,4,5,6,7,8,9,10,11,12,13,14,15,16,17,18,19,20,25,30,35,40,45,50,55,60,65,70,75,80,85,90,95,100} {
\node at (1-1/\y,\x/\y) {};
};
};
\node at (1,0) {};
\draw[thick] (-.3,2.3)--(1.5,0.5);
\fill[color=blue, fill opacity=0.1]
(-.3,2.3)--(-.5,2.1)--(1.3,.3)--(1.5,.5)--(-.3,2.3);
\draw[fill=gray, fill opacity=0.3, color=gray] (0,0)--(0,2)--(1,0)--(0,0);
\end{tikzpicture}
\caption{Slope $-1$ criterion at the corner $v_0$ of $\overline{\conv(\Gen)}$}
\label{fig:slope-1}
\end{figure}
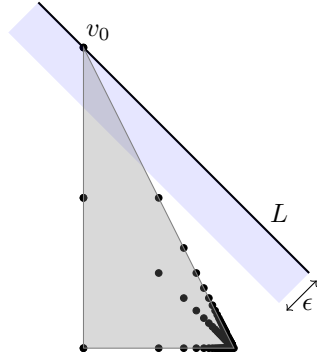
\begin{proof}
The set-up here is a special case of the corner criterion of Lemma \ref{lem:corner_fin_gen}, and the same proof holds to show finite generation of the diagonals.
We shall keep the same notation as in that proof, except that here the slope $(-1)$ condition on $L$ implies that the line equation $L(x,y)=0$ is of the special form $L(x,y)=ax+ay+c$.

Now, if $E_{p,q}[n+\bullet]$ is a summand in the $i$-th diagonal as expressed in \eqref{eq:diagonal}, then its height is
\[
\left|E_{p,q}[n+\bullet]\right| = a(p+q)+(n+\bullet)c = ai+nc + \bullet L(v_0) = ai+nc.
\]
In particular, the entire diagonal is built from summands with constant height $\lambda:= ai+nc = nL(i/n,0) = n\operatorname{dist}(L,(i/n,0))$. Note that the latter distance is equivalently described as the distance between the scaled line $nL$ through $nv_0$ and the diagonal $p+q = i$.

But in the proof of Lemma \ref{lem:corner_fin_gen} we showed that the only elements of height $\lambda$ are products of the GTCA $A_{v_0}$ and the finitely many products of generating subspaces $V_{i_1}\cdot\ldots\cdot V_{i_k}$ with $\sum n_{i_j} \leq \frac{\lambda}{\epsilon}$. It follows that the diagonals in question are each finitely-generated as an $A_{v_0}$-module, and with generators in the said degrees.
\end{proof}

\begin{example}[\textbf{Generation locus of orbit configuration spaces}]
\label{ex:genlocus}
The product formula in Theorem \ref{thm:ss_factorization}\eqref{eq:ss_product_factorization} provides a list of (free) generators for the $E^1$-page of the spectral sequence that computes $\Ho^{BM}_*(\Conf^\bullet_G(X,T))$:
\begin{itemize}
    \item $V_{i,n}:= \Ho^{BM}_i(X)\otimes \wt\Ho_{n-3}(\ol{\Q}_n) \leq E^{1}_{n-1,i}[n]$ for every $i\geq0$ and $n\geq 1$, and,
    \item $V_{[s],n}:= \wt\Ho_{n-2}(\ol{\D^T(G_s,s)}) \leq E^{1}_{n,0}[n]$ for orbit $[s]\in \F/G$ and every $n\geq 1$.
\end{itemize}
Letting $d$ be the degree of the top homology of $X$, the generation locus is thus
\[\Gen = \left\{\left(\frac{n-1}{n},\frac{i}{n}\right) \colon 0\leq i\leq d, n\geq 1\right\}\cup\{(1,0)\}.\]
Figure \ref{fig:genlocus} shows a plot of the generation locus, in the case that the top homology of $X$ is in degree $d=3$. Notice that, as $n$ approaches infinity, the points converge to the corner $(1,0)$.

Since for every fixed $\bullet=n$ there are only finitely many generating subspaces and each is finite-dimensional, the hypotheses necessary for the corner criterion hold. The two marked corners, at $(0,0)$ and $(0,d)$, in Figure \ref{fig:genlocus} are those isolated corners that give rise to finitely-generated module structures on $E^1_{*,*}$ according to Lemma \ref{lem:corner_fin_gen}.

Also notice that when $d=1$, the line of slope $(-1)$ through the corner $(0,1)$ passes through infinitely many points in the generation locus. 
In fact, the corner $(0,d)$ satisfies the conditions of Lemma \ref{lem:slope_1} if and only if $d\neq1$.

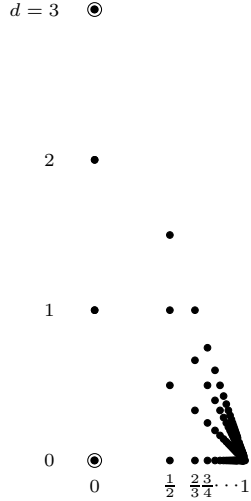
\begin{figure}[ht]
\begin{tikzpicture}[scale=2]
\node[draw,circle,minimum size=3pt,inner sep=0pt] at (0,3) {\scriptsize$\bullet$};
\node[draw,circle,minimum size=3pt,inner sep=0pt] at (0,0) {\scriptsize$\bullet$};
\node at (0,-1/6) {\scriptsize$0$};
\node at (1/2,-1/6) {\scriptsize$\frac{1}{2}$};
\node at (2/3,-1/6) {\scriptsize$\frac{2}{3}$};
\node at (3/4,-1/6) {\scriptsize$\frac{3}{4}$};
\node at (8/9,-1/6) {\scriptsize$\cdots$};
\node at (1,-1/6) {\scriptsize$1$};
\node at (-.3,0) {\scriptsize $0$};
\node at (-.3,1) {\scriptsize $1$};
\node at (-.3,2) {\scriptsize $2$};
\node at (-.4,3) {\scriptsize $d=3$};
\tikzstyle{every node}=[draw,circle,fill,minimum size=2.5pt,inner sep=0pt]
\node at (0,1) {};
\node at (0,2) {};
\foreach \x in {0,1,2,3} {
\foreach \y in {1,2,3,4,5,6,7,8,9,10,11,12,13,14,15,16,17,18,19,20,25,30,35,40,45,50,55,60,65,70,75,80,85,90,95,100} {
\node at (1-1/\y,\x/\y) {};
};
};
\node at (1,0) {};
\end{tikzpicture}
\caption{Generation locus for the $E^1$-page of the spectral sequence computing $\Ho^{BM}_*(\Conf^n_G(X,T))$; see Example \ref{ex:genlocus}.}
\label{fig:genlocus}
\end{figure}

\end{example}

The difference between the corner criterion \ref{lem:corner_fin_gen} and the slope $(-1)$ criterion \ref{lem:slope_1} becomes important when studying the homology of $\Conf^n_G(X,T)$: it may happen that every term $E^\infty_{(p,q)+\bullet v_0}$ by itself is a finitely-generated module, but there are infinitely many of them that contribute to the module $\Ho^{BM}_{i+\bullet v_0}$. This kind of behavior leads to the following definition.

\begin{definition}[\textbf{Filtered representation stability}]
Let $A[\bullet]$ be a GTCA and let $M[\bullet]$ be an $A$-module equipped with a filtration $F_*(M)$. There are two complementary notions of finite generation with respect to the filtration.
\begin{itemize}
    \item $M$ is \emph{bounded} finitely generated if for every bound $p$ on the filtration degree, $F_p(M)$ is a finitely generated $A$-module.
    \item $M$ is \emph{truncated} finitely generated if for every $p$, the quotient $M/F_p(M)$ is a finitely generated $A$-module.
\end{itemize}
\end{definition}
\begin{remark}[\textbf{Shifted filtration}]\label{rmk:shift_filtration}
It will often be the case that the module $M$ in the above definition is a GTCA, and $A$ is a subGTCA. If $A[\bullet]$ is concentrated in filtration degree $\ell\bullet$, one makes the $A$-module structure on $M$ compatible with the filtration by shifting $\wt{F}_p M[\bullet] := F_{p+\ell\bullet} M[\bullet]$. Then multiplication by $A$ preserves the filtration, and one can discuss filtered finite generation.
\end{remark}

\begin{lemma}[\textbf{Slope$\neq-1$ criterion for finite generation of diagonals}]
\label{lem:filtered_fg}
Let $E_{*,*}$ be again as in Lemma \ref{lem:corner_fin_gen} -- a bigraded GTCA generated by subspaces $V_i\leq E_{p_i,q_i}[n_i]$ such that for every $v\in \Gen$ and $n\in \NN$ the corresponding generating subspace \[\bigoplus_{\substack{v(V_i)=v\\ n_i=n}} V_i\] is finite dimensional.

Filter the diagonals by 
\[F_{\ell}\left(\bigoplus_{p+q=i} E_{p,q}\right) = \bigoplus_{\substack{p+q = i \\ p \leq \ell}} E_{p,q}.\]
Then for \emph{every} isolated corner $v_0=(r_0,t_0)$ of the closed convex hull $\ol{\conv(\Gen)}$, the $A_{v_0}$-module structures on the diagonals 
\[
    \bigoplus_{p+q = i+\bullet(r_0+t_0)}E_{p,q}[n+\bullet]
\]
are filtered finitely generated with respect to the shifted filtration $\wt{F}_{\ell} = F_{\ell+\bullet r_0}$.

More specifically, let $L$ be a non-vertical line that meets $\Gen$ only at $v_0$.
\begin{itemize}
    \item Suppose $\Gen$ lies below $L$. If $L$ has slope $m > -1$, then the diagonals are bounded finitely generated. Otherwise, if $m < -1$, then the diagonals are truncated finitely generated.
    \item If $\Gen$ lies above $L$, then the previous two cases are reversed.
\end{itemize}
Quantitatively, one could give explicit bounds on the values $\bullet$ at which generators might appear in terms the bounds in the corner criterion and the filtration degree $\ell$.
\end{lemma}
\begin{proof}
Let $L(x,y) = mx-y+c = 0$ be an equation of a line that intersects the convex hull of $\Gen$ only at the corner $v_0=(r_0,t_0)$. 
If $\Gen$ lies below $L$, then $L(v)\geq 0$ for all $v\in\Gen$ and $L(v_0)=0$. Otherwise, if $\Gen$ lies above $L$, then $L(v)\leq 0$ for all $v\in\Gen$ and $L(v_0)=0$.

Let us assume that $\Gen$ lies below $L$; the other case is completely analogous with all inequalities swapped.
As in the proof of Lemma \ref{lem:corner_fin_gen}, define the height of $e\in E_{p,q}[n]$ to be
\[
|e| = n\cdot L(v(e)) = mp-q+nc.
\]
Again, since the height is additive in all inputs $(p,q,n)$, it is additive under multiplication. And since it is nonnegative on a generating set for $E$, it follows that $|e|\geq 0$ for every element.

Now, consider the diagonal
\[
    \bigoplus_{p+q = i+\bullet(r_0+t_0)}E_{p,q}[n+\bullet].
\]
For an element $e\in E_{p,q}[n+\bullet]$ in this sum, 
\begin{align*}
0 & \leq mp-q+(n+\bullet)c & \text{since } |e|\geq 0\\
& = -(p+q) + (n+\bullet)c + (m+1)p\\
& = -(i+\bullet(r_0+t_0)) + (n+\bullet)c + (m+1)p & \text{since } e \text{ is on this diagonal}\\
& = -i + \bullet(-r_0-t_0+c) + nc + (m+1)p\\
& = -i + \bullet(-r_0-mr_0) + nc+ (m+1)p & \text{since } L(v_0)=0\\
& = -i + nc + (m+1)(p-\bullet r_0)
\end{align*}
which gives
\begin{equation}\label{eq:p_bound}
(m+1)(p-\bullet r_0) \geq i-nc.
\end{equation}

The diagonal is a direct sum of $A_{v_0}$-modules $E_{p_0+\bullet r_0,q_0+\bullet t_0}[n+\bullet]$. According to the inequality \eqref{eq:p_bound}, such a module contributes to the diagonal nontrivially only when
\[(m+1)p_0\geq i-nc.\]
When $m>-1$, this is equivalent to having a lower bound $p_0\geq \frac{i-nc}{m+1}$.
By additionally bounding the filtration degree by $\ell+\bullet r_0$, one also imposes a restriction $\ell\geq p_0$.
These bounds on $p_0$ leave only finitely many $A_{v_0}$-modules contributing to the $i$-th diagonal, each of which is finitely-generated by the corner criterion in Lemma \ref{lem:corner_fin_gen}. Thus, the diagonals are bounded finitely generated with respect to the shifted filtration $\wt{F}_\ell = F_{\ell+\bullet r_0}$.

On the other hand, if $m<-1$, inequality \eqref{eq:p_bound} gives the restriction $p_0\leq\frac{i-nc}{m+1}$. Then truncating the filtration would give an additional lower bound on $p_0$, thus again ensuring that only finitely many of these finitely generated $A_{v_0}$-modules contribute to the $i$-th diagonal. The diagonals are therefore truncated finitely generated.
\end{proof}

\subsection{Primary representation stability}
\label{sec:primarystab}
The $E^1$-page of the collision spectral sequence gives access to only limited information regarding $\Ho^{BM}_*(\Conf^n_G(X,T))$. From this point on we handle nonvanishing differentials by employing the powerful Noetherianity results for GTCAs. The primary, and most important example of this approach is given in Theorem \ref{thm:intro-primary_stability} and has already appeared in various special cases \cite{casto,petersen}.

\begin{theorem}[\textbf{Primary finite generation of homology}] \label{thm:primary-finite-generation}
Assume that $X$ is an almost free $G$-space following Conventions \ref{conventions} with $\dim \Ho_*^{BM}(X)<\infty$, and let $T\subset X$ be a finite $G$-invariant subset. 
Let $\Ho^{BM}_{d}(X)\neq 0$ be the top nonvanishing Borel-Moore homology group, and let \[A:= \Ind_{(1)}^{\FB} \Ho^{BM}_d(X)\] be the GTCA freely generated by $\Ho^{BM}_{d}(X)$.
The cross product \[\Ho^{BM}_{d}(X)\otimes \Ho^{BM}_{i}(\Conf^k_G(X,T)) \to \Ho^{BM}_{d+i}(\Conf^{k+1}_G(X,T))\] gives an action of $A$ on the $\FB$-modules of constant codimension $\Ho^{BM}_{d\bullet-i}(\Conf^\bullet_G(X,T))$ for every $i \geq 0$, preserving the collision filtration.

When $d\geq 2$, these $A$-modules are finitely generated. Explicitly, for every $i \geq 0$ there exist finitely many classes 
\[
\alpha_1,\ldots,\alpha_k \in \coprod_{n \in \NN} \Ho^{BM}_{dn-i}(\Conf^n_G(X,T))
\]
whose images under repeated multiplication by $\Ho^{BM}_{d}(X)$ generate  $\Ho^{BM}_{dm-i}(\Conf^m_G(X,T))$ as a $\G[m]$-representation for every $m\in \NN$.

Otherwise, when $d= 1$ these $A$-modules are \emph{bounded-}finitely generated relative to the collision filtration, i.e. every term 
\[
F_{p}\Ho^{BM}_{\bullet - i}(\Conf^\bullet_G(X,T))
\]
is finitely-generated.
\end{theorem}

\begin{proof}
From the compatibility of the collision filtration with GTCA multiplication, it is clear that the cross product with classes in $X$ respects the filtration. Passing to the associated spectral sequence, one identifies the sub-GTCA $A = \Ind^{\FB}_{(1)}\Ho_d^{BM}(X) \leq E^1_{*,*}$ in the product decomposition \ref{thm:ss_factorization}\eqref{eq:ss_product_factorization} as the one generated by this product operation.
Since all differentials point to the left, they must all vanish on $A\subseteq E^1_{0,d\bullet}$.
Thus the entire spectral sequence is in fact one of $A$-modules, and the $A$-action on homology is compatible with this structure.

We show by induction on $r$ that every diagonal
\[
\bigoplus_{p+q = d\bullet-i} E^r_{p,q}[\bullet]
\]
is a (bounded-) finitely generated $A$-module.
For the base case $r=1$, note that $A$ contributes the point $(0,d)\in \Gen(E^1)$ (see Figure \ref{fig:genlocus}) which is always an isolated corner separated from other points by a horizontal line -- this gives bounded-finite generation by Lemma \ref{lem:filtered_fg}. Furthermore, if $d\geq 2$ then the point $(0,d)$ is separated from the other points by a line of slope $(-1)$, thus by the slope $(-1)$ criterion in Lemma \ref{lem:slope_1} absolute finite generation follows.

For the induction step, since the differentials are $A$-linear, their kernel and cokernel are $A$-submodules. Therefore, for the induction to proceed it would suffice to know that the category of $A$-modules is locally Noetherian, i.e. that submodules of a finitely generated module are also finitely generated: then finite-generation would persist under computing the subquotients $E^{r+1} = \Ho(E^r,\partial_r)$. This Noetherianity property indeed holds as a result of the work of Sam-Snowden, explained next.

Consider the case $G=1$ first.
For any ring $R$ and $D\in \NN$, modules over $\Ind_{(1)}^{\FB[]} (R^D)$ are equivalent to representations of the category $\FI[D]$ (see \cite[Prop. 7.2.5.]{SS-Grobner}). Then \cite[Cor. 7.1.5]{SS-Grobner} is the claim that representations of $\FI[D]$ over a Noetherian ring are again Noetherian. Lastly, if $U$ is any finitely generated $R$-module, it receives a surjection $R^D\onto U$ and thus there is an induced map of TCAs $\Ind_{(1)}^{\FB[]} (R^D) \onto \Ind_{(1)}^{\FB[]} (U)$. Since the former TCA has a locally Noetherian category of representations, it follows by restriction that the latter TCA has the same property.

With a general finite group $G$, one can bootstrap from the TCA case in the previous paragraph. This proceeds by observing that the restriction of GTCAs to TCAs along the inclusion $1\leq G$ reflects finite generation of modules. Indeed, upon restriction one only has to consider all $G^n$-translates of generators, of which there would be finitely many.

Returning to our original problem of configuration spaces, Noetherianity of $A$-modules allows the induction to proceed. Since every $A$-module occurring in $E^1$ is (locally) finitely-generated, every later page $E^r$ would also be comprised of finitely-generated $A$-modules, and thus so will be $E^{\infty}$. Since the diagonal $\oplus_{p+q=d\bullet-i}E^{\infty}_{p,q}[\bullet]$ is the associated graded of $\Ho^{BM}_{d\bullet-i}$, they are all (bounded-)finitely generated as claimed.
\end{proof}

\begin{remark}[\textbf{Poincar\'{e} dual statement for manifolds}]\label{rem:duality}
When $X$ is an orientable $d$-manifold, its orbit configuration space $\Conf^n_G(X,T)$ is an orientable $dn$-manifold. Poincar\'{e} duality gives an identification
\[
\Ho^{BM}_{d\bullet-i}(\Conf^\bullet_G(X,T)) \cong \Ho^{i}(\Conf^\bullet_G(X,T))
\]
and the cross product with the fundamental class $[X]$ is conjugate to the ordinary pullback along the projection $\Conf_G^{n+1} \to \Conf_G^n$ forgetting a point.
Thus our Theorem \ref{thm:primary-finite-generation} recovers and extends classical representation stability for connected oriented manifolds.
\end{remark}

\begin{remark}[\textbf{Degrees of generators and relations}]
For finite generation to give explicit applications, one must get a handle on the degrees at which generators and relations appear. 
In Theorem \ref{thm:E1_finite_generation} below, we examine the $E^1$ page more throughly and obtain explicit bounds on degrees at which generators may appear.
We show that the
$A$-modules that make up the $E^1$-page are all free, generated in known degrees and satisfy no relations. Explicitly, the diagonal
\[\bigoplus_{p+q = d\bullet-i}E^1_{p,q}[\bullet]\]
is generated by elements in $E^1[n]$ with $n\leq \frac{i}{\epsilon}$ for $\epsilon=\min(\frac{d-1}{2},k)$ where $k\geq 1$ is least such that $\Ho^{BM}_{d-k}\neq 0$.

From this point, one would need to bound the effect differentials might have. This is standard practice in representation stability, and proceeds using the notions of injectivity and surjectivity degree (see e.g. \cite[\S3.1]{CEF} or \cite[\S4.2]{Wi-FIW}). We will not address this more quantitative aspect of stability in this work.
\end{remark}

Let us now relate the finite generation results of Theorem \ref{thm:primary-finite-generation} with representation theory, and constrain the irreducible decompositions of the sequence of representations $\G[n]\curvearrowright \Ho^{BM}_{dn-i}(\Conf^n_G(X,T))$. When the top homology $\Ho^{BM}_d(X)$ is an irreducible representation of $G$ (e.g. the trivial representation), and under the mild hypotheses of Theorem \ref{thm:primary-finite-generation}, there is a nice characterization for these irreducible decompositions. Similar descriptions are possible when $\Ho^{BM}_d(X)$ is reducible, but these get messy and hard to write down explicitly.

We will first recall some representation theory of the wreath product group $\G[n]$ (see \cite{kerber} for an exposition).
Given an $\sym$-representation $V$ and a $G$-representation $U$, denote $V[U]:=V\otimes U^{\otimes n}$, a representation of $\G$.

Suppose that $U_1,\dots,U_\ell$ is a complete list of irreducible representations of $G$. 
The irreducible representations of $\G[n]$ are characterized by the induction products $S^{\lambda_1}[U_1]\cdot S^{\lambda_2}[U_2]\cdots S^{\lambda_\ell}[U_\ell]$, 
where $\lambda_1,\dots,\lambda_\ell$ are integer partitions with $|\lambda_1|+\cdots+|\lambda_\ell|=n$ 
and $S^{\lambda_i}$ is the irreducible representation of $\sym[|\lambda_i|]$ corresponding to the partition $\lambda_i$.
Let us denote \[V(\lambda_1,\dots,\lambda_\ell):=S^{\lambda_1}[U_1]\cdot S^{\lambda_2}[U_2]\cdots S^{\lambda_\ell}[U_\ell].\]
Given an integer partition $\lambda=(a_1,\dots,a_k)$ and an integer $m\geq |\lambda|+a_1$, define a partition $\lambda\langle m\rangle = (m-|\lambda|,a_1,\dots,a_k)$. Considering the Young diagram of $\lambda$, this operation adds a top row to the diagram, obtaining a partition of $m$.

\begin{theorem}[\textbf{Multiplicity stability}] \label{thm:primary-multiplicity}
Assume that $X$ is an almost free $G$-space following Conventions \ref{conventions} with $\dim \Ho_*^{BM}(X)<\infty$, and let $T\subset X$ be a finite $G$-invariant subset.
Let $H^{BM}_{d}(X)\neq 0$ be the top nonvanishing homology group, and assume that $\Ho^{BM}_d(X)$ is an irreducible $G$-representation; without loss of generality we may assume $\Ho^{BM}_d(X)=U_1$. 

When $d\geq 2$, there exists a \emph{finite} set $\Lambda$ of $\ell$-tuples of partitions  
and positive integers $c(\underline{\lambda})$ for each $\underline{\lambda}=(\lambda_1,\dots,\lambda_\ell)\in\Lambda$ such that for all $n\gg 1$,
\[
\Ho^{BM}_{dn-i}(\Conf^n_G(X,T)) = \bigoplus_{\underline{\lambda}=(\lambda_1,\dots,\lambda_\ell)\in \Lambda} V(\lambda_1\langle n\rangle,\lambda_2,\dots,\lambda_\ell)^{\oplus c(\underline{\lambda})}.
\]
Furthermore, all $\underline{\lambda}=(\lambda_1,\dots,\lambda_\ell)\in\Lambda$ satisfy the bound $|\lambda_1|+\cdots+|\lambda_\ell|\leq i/\epsilon$ 
with $\epsilon=\min(\frac{d-1}{2},k)$ where $k\geq 1$ is least such that $\Ho^{BM}_{d-k}(X)\neq 0$.

When $d=1$ the same multiplicity stability result holds after bounding the collision filtration to degree $p$, now with partition bound $\sum |\lambda_j| \leq \frac{\sqrt2}{2}(i+2p)$.
\end{theorem}

\begin{proof}
From Theorem \ref{thm:primary-finite-generation} the module $\Ho^{BM}_{d\bullet-i}(\Conf^\bullet_G(X,T))$ is finitely generated under repeated multiplication by $\Ho^{BM}_d(X)\cong U_1$. Let $\{V_j\}_{j=1}^g$ enumerate the collection of generating irreducible representations, say $V_j$ occurs in configurations of $n_j$-points. Explicitly, for every $n\geq 1$ the $\G$-representation $\Ho^{BM}_{dn-i}(\Conf^n_G(X,T))$ is a quotient of the sum of the induction products
\begin{equation}\label{eq:induction_product}
    V_j \boxtimes (\underbrace{U_1\boxtimes\ldots \boxtimes U_1}_{k \text{ times}}) = V_j\cdot S^{(k)}[U_1]
\end{equation}
where $S^{(k)}$ is the trivial representation of $\sym[k]$ and $k= n-n_j$.

By the classification of irreducibles $V_j= S^{\mu_1}[U_1]\cdot S^{\mu_2}[U_2]\cdots S^{\mu_\ell}[U_\ell]$ for some partitions satisfying $|\mu_1|+\ldots+|\mu_{\ell}|=n_j$. Thus the induction product in \eqref{eq:induction_product} is
\[
\left(S^{\mu_1}[U_1]\cdot S^{\mu_2}[U_2]\cdots S^{\mu_\ell}[U_\ell]\right)\cdot S^{(k)}[U_1] 
\cong (S^{\mu_1}\cdot S^{(k)})[U_1]\cdot S^{\mu_2}[U_2]\cdots S^{\mu_\ell}[U_\ell].
\]
The branching rules for $S^{\mu_1}\cdot S^{(k)}$ work just as for symmetric groups, and hence the decomposition of such a product stabilizes as in ordinary representation stability for $\FI[]$-modules. More explicitly, \cite[Proposition 3.1.3]{SS-FIG} explains that when $U_1$ is the trivial representation, a module as we have above is equivalent to a representation of the category $\FI[]\times\FB[]^{\ell-1}$, with every factor acting on a corresponding term in the induction product. In particular, the multiplicities associated with $U_1$ follow the usual pattern of multiplicities in $\FI[]$ modules. But a quick inspection of their proof shows that it carries over to having $U_1$ be any irreducible of $G$.

Lastly, when $d\geq 2$ Lemma \ref{lem:slope_1} gives a bound on the generator degree of the $E^1$-page of the collision spectral sequence: $n_j \leq i/\epsilon$, and when $d=1$ one gets the filtration dependent bound from the proof of Lemma \ref{lem:filtered_fg}. These bound the partition lengths appearing on $E^1$, and therefore also the ones in homology.
\end{proof}

Our setup gives rise to symmetric group representations in two different ways: either by restricting the action $\G \curvearrowright \Conf_n^G(X,T)$ to $\sym\leq \G$, or by forgetting the $G$ action on $X$ altogether and considering the ordinary configuration space $\Conf_n(X,T)$ with its symmetric group action.

In both cases, the algebraic structure that arises in Theorem \ref{thm:primary-finite-generation} is what's known in the literature as an $\FI[D]$-module, where $D$ is the dimension of the top nonvanishing Borel-Moore homology $\Ho^{BM}_d(X)$.
These are representations of a certain category $\FI[D]$ described explicitly in \cite{SS-Grobner}, whose finitely generated representations have been studied extensively by Ramos \cite{Ramos}. Their work implies stability patterns in the representations $\sym\curvearrowright \Ho_{dn-i}^{BM}(\Conf_n^G(X,T))$, mainly injectivity and surjectivity properties of the maps $\Ho_{dn-i}^{BM} \to \Ho_{d(n+1)-i}^{BM}$ and constraints on their irreducible decompositions (see \cite[Theorem A]{Ramos} for details).
In particular, we obtain the following statement using \cite[Theorems B and 2.15]{Ramos}.
\begin{corollary}[\textbf{Restricting to the symmetric group action}]
Assume that $X$ is an almost free $G$-space following Conventions \ref{conventions} with $\dim \Ho_*^{BM}(X)<\infty$, and let $T\subset X$ be a finite $G$-invariant subset.
Let $\Ho_d^{BM}(X)\neq0$ be the top nonvanishing homology, and set $D = \dim\Ho_d^{BM}(X)$. For $d\geq 2$ and any $i\geq 0$, consider $\Ho_{dn-i}^{BM}(\Conf_n^G(X,T))$ as a representation of the symmetric group $\sym$. 
\begin{enumerate}
\item
For any partition $\lambda$, there exists a polynomial $P_\lambda(t)$ with $\deg P_\lambda<D$ such that the multiplicity of the irreducible $\sym$-representation $V(\lambda\langle n\rangle)$ in $\Ho_{dn-i}^{BM}(\Conf_n^G(X,T))$ is equal to $P_\lambda(n)$ for all $n\gg1$.
\item
There exist polynomials $p_1(t),\dots,p_D(t)$ such that for all $n\gg1$,
\[
\dim \Ho^{BM}_{dn-i}(\Conf^G_n(X,T)) = p_1(n)+ p_2(n)2^n + \ldots + p_D(n) D^n.
\]
\end{enumerate}
When $d=1$, the same holds after bounding the collision filtration to degree $p$.
\end{corollary}

\subsection{Secondary and higher representation stability}
In \cite{MW}, Miller and Wilson discovered a phenomenon of secondary representation stability: for a manifold with boundary $M$, their construction gives rise to a stabilization map \[\Ho_i(\Conf^n(M))\to \Ho_{i+1}(\Conf^{n+2}(M))\] by introducing a pair of orbiting points near the boundary of $M$. They then show that after ``factoring out'' a primary stabilization given by introducing one point near the boundary (whose Poincar\'{e} dual map is a right inverse to our primary stabilization -- more details in \S\ref{subsec:orbiting_pair}) their new stabilization map gives isomorphisms with improved stable range. Following their example, this section explores secondary operations on the sequence of representations $\Ho^{BM}_{*}(\Conf_G^{\bullet}(X,T))$ that become important after factoring out the primary stabilization action of the previous section, though we prove new theorems only in the special case of $i$-acyclic spaces, for which the collision spectral sequence already collapses at the $E^1$-page. Examples to which our analysis will apply are the affine and toric root system arrangements, which in large part motivated this project.

When the collision spectral sequence collapses at $E^1$ one essentially gets a formula for the homology as a GTCA (up to an extension problem). In this regard invariants such as Betti numbers and multiplicities of irreducible representations are in principle completely computable, though some qualitative questions are still hard to answer in practice. The main purpose of this section is therefore different: we seek to understand the various stabilization operations and quantify finite generation under them. One potential broader application of this analysis is to the study of ``derived generators", also called $\mathtt{FI}$-hyperhomology in some contexts, see Remark \ref{rmk:FI-homology} for more details.

Consider the many module structures on the $E^1$-page of our spectral sequence that arise from the product factorization \eqref{eq:ss_product_factorization} in Theorem \ref{thm:ss_factorization}. For efficiency of notation, let us break somewhat from the notation of \S\ref{subsec:generation_locus} and write 
\[
A^n_i := \Ind_{G\times \sym}^{\FB} \Ho^{BM}_i(X)\boxtimes \wt\Ho_{n-3}(\ol{\Q}_n)
\]
for the term that appears in the $n$-th factor of \eqref{eq:ss_product_factorization}. Note that this is precisely the subalgebra of $E^1$ generated by $V_{i,n}$ from Example \ref{ex:genlocus}, and $A_d^1$ is the algebra of primary stabilization operations studied in \S\ref{sec:primarystab}.

Now, every pair $(i,n)$ gives an action by multiplication 
\[A^n_i[\bullet]\otimes E^1_{p,q}[m]\to E^1_{p+(n-1)\frac{\bullet}{n},q+i\frac{\bullet}{n}}[m+\bullet]\]
making the $E^1$-page into a direct sum of $A^n_i$-modules labeled by triples $(p,q,m)$. In fact, the product factorization in \eqref{eq:ss_product_factorization} shows that every one of these modules is \emph{free}.

\begin{definition}[\textbf{Factoring out an action}]
\label{def:factor_out_action}
For any GTCA $A$ and any $A$-module $E$, extracting the generators for $E$ amounts to computing the quotient by the ``augmentation ideal"
\[
(E/{A_{>0}E})[\bullet] = E[\bullet]/\left( \sum_{\substack{i+j=\bullet \\ i>0}} A[i]E[j] \right).
\]
We refer to this operation as \emph{factoring out the $A$-action}.
\end{definition}

\begin{remark}[\textbf{Factoring out is zeroth $\FI[]$-homology}]\label{rmk:FI-homology}
This operation of `factoring out' and its derived functors are ubiquitous in the representation stability literature (see e.g. \cite{church-ellenberg, ramos-li}). For example, in the case that $G=1$ and $A=\mathds{1}_\bullet$, an $A$-module is nothing but an $\FI[]$-module. Then factoring out the $A$-action is precisely what \cite[Def. 2.3.7.]{CEF} calls $\FI[]$-homology $\Ho^{\FI[]}_0(-)$.

Note that when working with rational coefficients, all modules appearing in our $E^1$-page are projective. It follows that factoring out actions at the level of $E^1$ computes the associated derived functor -- a generalization of $\FI[]$-hyperhomology. In particular, all results below about vanishing ranges and generator degrees for $E^1$ imply the same vanishing ranges for the derived factoring-out functors. For $\FI[]$-modules such vanishing results have been translated back to stable range calculations by Gan-Li \cite{GL-linear-range}, and have been utilized in the case of configuration spaces of closed manifolds by Miller-Wilson \cite{MW2}).

We should also point out recent work by Ho \cite{Ho} which obtains essentially the same bounds presented below using a completely different approach via factorization homology, though he only considers rational coefficients. See \cite{Ho} for a more complete discussion of the derived factoring-out functors and their vanishing ranges in the context of configuration spaces.
\end{remark}

\begin{example}[\textbf{Factoring-out for free GTCAs}]
\label{ex:factor_out_action}
For a free module $E = A\otimes V$, one clearly gets
\[
A\otimes V/(A_{>0}\otimes V) \cong V
\]
since $A$ is unital.
This could be understood as formally deleting the `$A\otimes$' factor from the product factorization of $E$. 

In the context of our GTCAs $A^n_i$ acting on $E^1$, a quotient $E^1/{(A^n_i)_{>0}E^1}$ simply removes the term $\Ho^{BM}_i(X)\otimes \wt\Ho_{n-3}(\ol{\Q}_n)$  
appearing in the product \eqref{eq:ss_product_factorization}. Furthermore, since \eqref{eq:ss_product_factorization} also implies that the various $A^n_i$-actions on $E^1$ commute (possibly up to signs), one can still define an $A^{m}_{j}$-action on any quotient $E^1/(A^n_i)_{>0}E^1$ and these remain projective $A^m_j$-modules. This observation will enter our discussion of secondary and higher stability below.
\end{example}

With the understanding that representation stability is interpreted as finite generation of modules over GTCAs, further actions of GTCAs on the spaces of generators for $E^1$ exhibiting them as finitely generated modules is thus understood as secondary representation stability. The next result manifests this procedure to uncover many layers of finitely-generated module structures on $E^1$.

\begin{theorem}[\textbf{Finite generation of diagonals in $E^1$}]\label{thm:E1_finite_generation}
Assume that $X$ is an almost free $G$-space following Conventions \ref{conventions} with $\dim \Ho_*^{BM}(X)<\infty$, and let $T\subset X$ be a finite $G$-invariant subset.
Let $E^1_{*,*}$ be the GTCA computing $\Ho^{BM}_*(\Conf^n_G(X,T))$, and for each pair $(i,n)$ let $A_i^n$ be the subGTCA freely generated by $\Ho^{BM}_i(X)\otimes\wt\Ho_{n-3}(\Q_n)$. 
The diagonals in $E^1_{*,*}$ admit multiple structures of a (bounded) finitely-generated module over the subGTCAs $A^n_i$, obtained by the following iterative procedure.

Set $E=E^1_{*,*}$ and let $\Gen(E)$ be the generation locus of $E$ (see Example \ref{ex:genlocus} and Figure \ref{fig:genlocus}).
\begin{enumerate}
    \item Find the points $v(A^n_i) = \left(\frac{n-1}{n},\frac{i}{n}\right) \in \Gen(E)$ of maximal taxi-cab norm.
    
    If $v:=v(A^n_i)$ is the unique point in $\Gen(E)$ 
    of maximal norm, then the diagonals

    \begin{equation}\label{eq:diagonal_module}
        \bigoplus_{p+q=\|v\|\bullet-j} E_{p,q}[m+\bullet]
    \end{equation}
    form a finitely-generated free $A^n_i$-module for every pair $(j,m)\in \NN^2$. Moreover, all generators appear in $E[k]$ for $k \leq \frac{m\|v\|+j}{\epsilon}$ where $\epsilon$ is the difference between the two largest taxi-cab norms in $\Gen(E)$.

    Otherwise, if there are multiple points of maximal norm, pick the one with minimal $x$-coordinate, say $v(A^n_i)\in \Gen(E)$. Then the free $A^n_i$-module in \eqref{eq:diagonal_module} is only bounded finitely generated with respect to the appropriately shifted filtration as in Remark \ref{rmk:shift_filtration}).
    
    \item A description of the (free) generators is attained by factoring out the action of $A^n_i$.
    By replacing $E$ with the quotient $E/(A_i^n)_{>0}E$ and returning to Step 1, the space of generators will itself admit further structures of finitely-generated free modules.
\end{enumerate}
\end{theorem}

\begin{remark}[\textbf{Geometric view of Theorem \ref{thm:E1_finite_generation}}]
Before the proof, we describe the procedure geometrically.
One starts by sweeping in a line of slope $(-1)$ from far above $\Gen(E)$. The first point in $\Gen(E)$ that the line hits will give rise to finitely-generated module structures on diagonals. If the line first hits multiple points, then the one
farthest to the left gives bounded finite generation.
Then factor out the action of this extremal point and iterate the procedure.
\end{remark}

\begin{proof}
    First note that the only accumulation point of $\Gen(E^1_{*,*})$ is $(1,0)$. Away from this point, $\Gen(E^1_{*,*})$ is a bounded discrete set. Thus when applying the geometric criteria for finite generation given in Lemmas   \ref{lem:slope_1} and \ref{lem:filtered_fg}, one need not worry about separating a line from $\Gen(E)$ as long as the line does not pass through $(1,0)$. 
    The same reasoning applies to $E$ at any iterate of the procedure, as the set $\Gen(E)$ is a subset of $\Gen(E^1_{*,*})$.
    
    Now suppose that $E_{*,*}$ is given at any iterate of the procedure. Then by Example \ref{ex:factor_out_action}, $E_{*,*}$ is still freely generated by a subcollection of the GTCAs $\{A^n_i\}_{(n,i)}$.
    
    Recall that in the taxi-cab metric, a sphere around $(0,0)$ intersected with the first quadrant is a line of slope $(-1)$. This implies that the line of slope $(-1)$ passing through one of the points in $\Gen(E)$ with maximal taxi-cab norm must pass through all points in $\Gen(E)$ with maximal norm. Denote this line by $L$.

    Let $v\in\Gen(E)$ be the point on $L$ with least $x$-coordinate
    (minimization of $x$ happens on the discrete part of $\Gen(E)$ and thus a minimum exists). Then any line $L'$ through $v$ with slope $-1+\epsilon$
    passes through $v$ alone if $\epsilon>0$ is sufficiently small, and avoids $(1,0)$. Thus by the corner criterion in Lemma \ref{lem:filtered_fg}, every diagonal becomes a finitely generated $A_v$-module after bounding the (shifted) collision filtration degree.
    
    If, in addition, the point $v\in \Gen(E)$ is the unique point of maximal norm, then the slope $(-1)$ criterion in Lemma \ref{lem:slope_1} applies and shows that the $A_v$-action on the diagonals is already finitely-generated. As for the claimed bound on generator degrees, the slope criterion also gives the bound $n\leq \dist(mL,(j,0))/\epsilon$. But the distance between a line of slope $(-1)$ and a point, both in the first quadrant, is the difference of their taxi-cab norms divided by $\sqrt{2}$.
    Since both the numerator and the denominator are given by such distances, the roots cancel. A quick check shows that the same distance formula holds when $j$ is negative.
\end{proof}

\begin{remark}
The reader is warned that the $A^n_i$-module structures on the $E^1$-page are \emph{not} in general compatible with the differentials, except for the cases with $n=1$ -- stabilizations by adding one point moving along a cycle. The problem is of course that $A^n_i \leq E^1$ might not itself lie in the kernel of all differentials, and thus will not commute with them.

One should therefore not expect to see these actions at the level of homology, or even on any page $E^r$ with $r>1$. However, since a spectral sequence gives a sort of an upper bound on homology, one can still extract representation theoretic information without any
further assumptions (see e.g. Theorem \ref{thm:primary-multiplicity}).
\end{remark}

Theorem \ref{thm:E1_finite_generation} gives concrete information about homology in many special cases. Mainly, when the collision spectral sequence collapses already at the $E^1$-page: in this case the GTCA $E^1$ essentially gives a formula for the Borel-Moore homology of the orbit configuration space, up to a possible extension problem. However, Theorem \ref{thm:E1_finite_generation} unpacks the product decomposition and converts it into a quantitative statement about degrees of generators, which is difficult to see in the product decomposition of Theorem \ref{thm:ss_factorization} directly.

A central class of examples in which this collapse occurs is called \emph{$i$-acyclic spaces}: spaces $X$ in which the natural map $\Ho_*(X) \to \Ho^{BM}_*(X)$ is trivial (see \cite{Arabia,petersen-formal}). These spaces include all orientable manifolds on which the cup product restricts to zero on compactly supported cohomology $\Ho^*_c(X)$ -- in particular the cases of affine and toric arrangements (when $X$ is $\C$ or $\C^\times$).

\begin{corollary}[\textbf{Higher finite generation for $i$-acyclic spaces}]\label{cor:i-acyclic}
Suppose that $X$ is an $i$-acyclic space, i.e. $\Ho_c^*(X)\to \Ho^*(X)$ vanishes, such as an orientable manifold with trivial cup product on $\Ho^*_c(X)$ or any space of the form $X'\times \RR$. Assume further that $\dim \Ho^{BM}_*(X) < \infty$
(see also Conventions \ref{conventions}), and let $T\subset X$ be a finite $G$-invariant subset.

Then the various GTCAs $A^n_i$, freely generated by
$\Ho^{BM}_i(X)\otimes \wt\Ho_{n-3}(\ol\Q_n)$, act on the associated graded homology $\operatorname{gr}^F\Ho^{BM}_*(\Conf^\bullet_G(X,T))$ giving rise to a sequence of (bounded) finitely generated free module structures, by following the procedure of Theorem \ref{thm:E1_finite_generation}.

Most significantly, let $\Ho^{BM}_d(X)\neq 0$ denote the top nonzero homology. Then for every $k < \frac{d}{2}-1$ stabilizing by cross product with $\Ho^{BM}_{d-k}(X)$ gives free module structures on $\Ho^{BM}_{(d-k)\bullet-j}(\Conf^\bullet_G(X,T))$ finitely-generated from classes at which $\bullet < j$ after factoring out the actions of $\Ho^{BM}_{d-i}(X)$ for $0\leq i < k$.

If $d$ is even, the same finite generation results holds for $k=\frac{d}{2}-1$ but with generators only bounded by $\bullet\leq 2j$. When $d$ is odd, multiplication by $\Ho^{BM}_{\frac{d+1}{2}}(X)$ only gives finitely generated free modules at bounded collision filtration degree.
\end{corollary}
\begin{proof}
Petersen shows in \cite{petersen-formal} that the collision spectral sequence collapses at $E^1$ for $i$-acyclic spaces. Thus every term $\Ho^{BM}_i(X)\otimes \wt\Ho_{n-3}(\ol\Q_n) \subseteq E^1_{n-1,i}[n]$ is isomorphic to a subspace of $\Ho^{BM}_{i+n-1}(\Conf^n_G(X,T))$ under a choice of splitting of the collision filtration. In this way, the GTCA multiplication on homology equips it with an action of $A^n_i$, compatible with the $A^n_i$-module action on $E^1$ up to extensions.
Now, since finite generation persists under extensions, the homology will be finitely-generated whenever the diagonals of $E^1$ are so. The latter cases are described in Theorem \ref{thm:E1_finite_generation}.

In particular, for every $k < \frac{d}{2}-1$, the $k$-th iterate of the procedure picks the point of maximal taxi-cab norm $(0,d-k)\in \Gen(E)$. This corresponds to the action of multiplication by $\Ho^{BM}_{d-k}(X)$, now seen to freely generate $\Ho^{BM}_{(d-k)\bullet-j}(\Conf^\bullet_G(X,T))$ modulo the actions by $\Ho^{BM}_{d-i}(X)$ for $i<k$. The generators for this action are constrained to lie in $\bullet \leq j$.

The $k=\lceil\frac{d}{2}-1\rceil$-th step in the iterative process of Theorem \ref{thm:E1_finite_generation} again picks out stabilization by $\Ho^{BM}_{d-k}(X)$, but differs between even and odd $d$. In the even case, the point of $\Gen(E)$ with second smallest norm is $\left(\frac{1}{2},\frac{d}{2}\right)$ at distance $\epsilon=1/2$, thus giving the worse bound on generators. In the odd case, this point has the same maximal norm $d-k$, and therefore stabilization by $\Ho^{BM}_{d-k}(X)$ only gives bounded finite generation.
\end{proof}

\begin{example}\label{ex:high_dimensions}
The above kind of secondary and higher stability is better behaved and is more interesting when the homological dimension $d$ is large, e.g. for manifolds of large dimension $d$. For example, when
\begin{description}
    \item[$d=1$] even the primary stabilization only gives finitely generated modules after bounding the collision filtration degree.
    \item[$d=2$] primary stabilization means that the modules  $\Ho^{BM}_{d\bullet-j}(\Conf^\bullet_G(X,T))$ are finitely generated by classes in configurations of $\leq 2j$ points.
    \item[$d=3$] primary stabilization gives finite generation of $\Ho^{BM}_{d\bullet -j}(\Conf^\bullet_G(X,T))$ with improved bounds on generation degrees, now coming from configurations of $\leq j$ points. There is also secondary stabilization by $\Ho^{BM}_{d-1}(X)$, but it only gives bounded finite generation.
    \item[$d=4$] now the secondary stabilization by $\Ho^{BM}_{d-1}(X)$  gives finite generation with generators coming from configurations of $\leq 2j$ points. 
    \item[$d=5$] the secondary stabilization operation gives finite generation with the improved bound of $\leq j$ on generation degrees. Now there is also tertiary stabilization by $\Ho^{BM}_{d-2}(X)$  giving bounded finite generation.  
\end{description}
and this process continues in the predictable way.
\end{example}

\begin{remark}[\textbf{Improving the stable range}]\label{rmk:high-stability}
The intended way to conceptualize Theorem \ref{thm:E1_finite_generation} and Corollary \ref{cor:i-acyclic} is as follows. One thinks of the actions of $A^n_i$ as various stabilization operations, comparing $\Conf^k$ to $\Conf^{k+n}$. The first of these stabilizations mentioned in Theorem \ref{thm:E1_finite_generation} is multiplication by $\Ho^{BM}_{top}(X)$ -- this is the primary stabilization action of \S \ref{sec:primarystab}, and the finite generation of Theorem \ref{thm:primary-finite-generation} is interpreted as \emph{primary representation stability}.

Then the statement about finite generation for the later actions should be understood as secondary stability -- showing that there is yet another finite list of generators that strictly improves the range of generation. This improvement is quantified by saying that the \emph{slope} of stability has decreased.

\begin{definition}[\textbf{Slope}]
We say that, for a fixed pair $(j_0,n_0)$, the sequence of homology groups $\Ho^{BM}_{j_0+k\bullet}(\Conf_G^{n_0+\ell \bullet}(X,T))$ has slope $\frac{k}{\ell}$. Graphically, this is the slope of the line passing through the groups when drawn on a grid with \[(i,j) \mapsto H^{BM}_{i}(\Conf_G^j(X,T)).\]
\end{definition}

A quantitative goal of representation stability is to control the behavior of these sequences lying on slopes as small as possible.
In our picture of the generation locus $\Gen(E)$ (see Figure \ref{fig:genlocus}), the slope of diagonals under an $A_v$-action is precisely the taxi-cab norm $\| v\|$ -- hence the minimization process of Theorem \ref{thm:E1_finite_generation}.

For experts more familiar with (co)homological stability of configuration spaces of manifolds, under Poincar\'{e} duality our slope $d-s$ corresponds to the classical slope $s$ for a $d$-manifold. In particular, our attempt to understand Borel-Moore homology at small slopes corresponds classically to working with large slopes.

Now one can understand Theorem \ref{thm:intro-secondary_stability} as a restatement of Corollary \ref{cor:i-acyclic} using the slope terminology: allowing a larger GTCA to act on homology reduces the generation slope as follows. 
\begin{proof}[Proof of Theorem \ref{thm:intro-secondary_stability}]
In Corollary \ref{cor:i-acyclic} we show that the module $\Ho^{BM}_{(d-k)\bullet-j}(\Conf^\bullet_G(X,T))$ is generated under multiplication by $\Ho^{BM}_{d-k}(X)$ with generators satisfying $\bullet<j$, modulo the image of multiplications by $\Ho^{BM}_{d-i}(X)$ for $0\leq i <k$. Equivalently, $\Ho^{BM}_{(d-k)n-j}(\Conf^n_G(X,T))$ is generated by the image of multiplication by $\Ho^{BM}_{d-i}(X)$ for $0\leq i \leq k$ whenever $j<n$.

Write $j'=kn+j$, then one has a surjection onto $\Ho^{BM}_{dn-j'}(\Conf^n_G(X,T))$ whenever $n>j = j'-kn$, or $j' < (k+1)n$ as claimed. 
\end{proof}

\end{remark}

Let us remark on the case of a general space $X$ for which differentials might not vanish, and study higher stability in its configuration space. After passing to a subquotient of $E^1$, the bounds we give on degrees of generators may no longer apply, and without them it is not clear whether surjectivity at decreased slopes as in Theorem \ref{thm:intro-secondary_stability} holds in any form. We believe that properly addressing the homological algebra of the derived `factoring-out' functors will likely reproduce analogous results in general, especially in light of the fact that the $E^1$-page is built from free modules that are acyclic for factoring out at least for $\QQ$-coefficients (see Remark \ref{rmk:FI-homology}). 

\begin{conjecture}[\textbf{Proposed high dimensional secondary stability}]
Assume that $X$ is an almost free $G$-space following Conventions \ref{conventions} with $\dim \Ho_*^{BM}(X)<\infty$, let $T\subset X$ be a finite $G$-invariant subset, 
and let $H^{BM}_{d}(X)\neq 0$ be the top nonvanishing homology group.

Then for all $k < \frac{d-1}{2}$, the homology $\Ho^{BM}_{(d-k)\bullet-j}(\Conf^\bullet_G(X,T))$ modulo the representations generated by cross product actions \[\Ho^{BM}_{d-i}(X)\otimes \Ho^{BM}_*(\Conf^\bullet_G(X,T)) \to \Ho^{BM}_{*+d-i}(\Conf^{\bullet+1}_G(X,T))\] for all $i=0,\ldots,k-1$, satisfies the same finite generation results under multiplication by $\Ho^{BM}_{d-k}(X)$ as the primary stabilization in Theorems \ref{thm:primary-finite-generation}.

If $k=\frac{d-1}{2}$ then finite generation after factoring out the previous actions holds only at bounded collision filtration degrees.
\end{conjecture}
These are stability statements for respective slope $d-k$ (see Remark \ref{rmk:high-stability}), which under Poincar\'{e} duality translates to slope $k$ in ordinary cohomology.

For experts we suggest that a central technical result that could prove the above conjecture is the generalization of \cite[Theorem 5]{GL-linear-range} from $\FI[]$-modules to $\FI[d]$-modules, i.e. for modules over TCAs generated by more than one element.

\subsection{Stabilization by an orbiting pair} \label{subsec:orbiting_pair}
The stabilizations in Corollary \ref{cor:i-acyclic} above come from the first $\lfloor\frac{d-1}{2}\rfloor$ iterations of the procedure in Theorem \ref{thm:E1_finite_generation}. Pushing the procedure past this point and stabilizing by the top homology term associated with $n=2$ in Factorization \eqref{eq:ss_product_factorization} reveals a new phenomenon, akin to Miller--Wilson's secondary stability by introducing a pair of orbiting points \cite{MW}.
In this section, we need the following hypotheses:
\begin{hypothesis}[\textbf{Acyclic diagonal}]\label{hyp:vanishing_diagonal}
Let $X$ be a $G$-space with $\dim\Ho^{BM}_*(X)<\infty$, and let $\Ho^{BM}_d(X)\neq0$ be the top nonvanishing homology group. 
Further assume that the diagonal $\Delta_*:\Ho^{BM}_d(X)\to \Ho^{BM}_d(X^2)$ is zero on the top homology (e.g. $i$-acyclic orientable manifolds, see \cite{petersen-formal}). 
\end{hypothesis}

To see the connection between the next stabilization operation and the introduction of a pair of orbiting points, consider first the case of an oriented $d$-manifold $M$ with $G$ trivial and the class of two closely-orbiting points in $\Ho_{d-1}(\Conf^2(M))$. In Borel-Moore homology, a Poincar\'{e}-dual to this class is represented by any Borel-Moore coboundary of the diagonal, i.e. any $(d+1)$-chain in $M^2$ whose boundary is supported on the diagonal and represents its fundamental class $[\Delta]$, as explained next.

For $M=\RR^d$ one has $\Conf^2(\RR^d) \simeq S^{d-1}$. Thus $\Ho_*(\Conf^2(\RR^d))$ is concentrated in degree $d-1$ and generated by the class of orbiting points. The long exact sequence of the inclusions $\Conf^2(\RR^d) \subset \RR^{2d}\supset \Delta$:
\begin{center}
\begin{tikzpicture}
\node (0) at (0.4,1) {$\dots$};
\node (1) at (2.25,1) {$\Ho_{d+1}^{BM}(\RR^{2d})$};
\node (2) at (5.25,1) {$\Ho_{d+1}^{BM}(\Conf^2(\RR^d))$};
\node (3) at (8,1) {$\Ho_d^{BM}(\Delta)$};
\node (4) at (10.35,1) {$\Ho_d^{BM}(\RR^{2d})$};
\node (5) at (12.2,1) {$\dots$};
\draw[-angle 90] (0) to (1);
\draw[-angle 90] (1) to (2);
\draw[-angle 90] (2) to (3);
\draw[-angle 90] (3) to (4);
\draw[-angle 90] (4) to (5);
\node (a) at (2.25,0) {$0$};
\node (b) at (5.25,0) {$[c]$};
\node (c) at (8,0) {$[\Delta]$};
\node (d) at (10.35,0) {$0$};
\draw[|-angle 90] (2.75,0) to (4.75,0);
\draw[|-angle 90] (5.75,0) to (7.5,0);
\draw[|-angle 90] (8.5,0) to (9.85,0);
\end{tikzpicture}
\end{center}
shows that $\Ho^{BM}_{d+1}(\Conf^2(\RR^d))$ is generated by any chain in $\RR^{2d}$ whose boundary is supported on the diagonal and $\partial(c)=[\Delta]$. Therefore, such a class $[c]$ is the unique dual to the class of orbiting points.

For a general $d$-manifold $M$, consider again the long exact sequence of the inclusions $\Conf^2(M)\subset M^2\supset\Delta$. Our assumption that $\Delta_*$ is zero on top homology precisely guarantees that there exists a chain $c\in C^{BM}_{d+1}(M^2)$ whose boundary is supported on the diagonal and $\partial(c)=[\Delta]$, by exactness:
\begin{center}
\begin{tikzpicture}
\node (0) at (3,1) {$\dots$};
\node (2) at (5.25,1) {$\Ho_{d+1}^{BM}(\Conf^2(M))$};
\node (3) at (8,1) {$\Ho_d^{BM}(\Delta)$};
\node (4) at (11.35,1) {$\Ho_d^{BM}(M^{2})$};
\node (5) at (13.2,1) {$\dots$};
\draw[-angle 90] (0) to (2);
\draw[-angle 90] (2) to (3);
\draw[-angle 90] (3) to node[above] {$\Delta_*$} (4);
\draw[-angle 90] (4) to (5);
\node (b) at (5.25,0) {$[c]$};
\node (c) at (8,0) {$[\Delta]$};
\node (d) at (11.35,0) {$[\Delta]=0$};
\draw[|-angle 90] (5.75,0) to (7.5,0);
\draw[|-angle 90] (8.5,0) to (10.5,0);
\end{tikzpicture}
\end{center}
To see that $[c]$ is dual to the class of orbiting points in $M$, first consider a small ball $\RR^d \cong U\subset M$. Passing to configurations $\Conf^2(\RR^d) \to \Conf^2(M)$, the class of closely orbiting points in $M$ is the image of a similar class for $U$. Since $\Delta_M \cap U^2 = \Delta_U$, it follows that $c$ restricts to the Poincar\'e dual class of the orbiting pair in $U$. A push-pull formula for the restriction from $\Conf^2(M)$ to $\Conf^2(U)$ shows that $c$ indeed gives a dual to the orbiting pair in $M$ as well.

 Geometrically, one can think of a chain in $M^2$ whose boundary is the diagonal as prescribing a coherent trajectory of two distinct points from being coincident to leaving $M$ through the boundary. Such a trajectory pairs with the class in which the two points orbit each other precisely once.

In our $E^1$ page, a chain representing a dual for the orbiting pair is represented by the fundamental class
\[
[M] \in \Ho^{BM}_{d}(M) \cong E^1_{1,d}[2].
\]
The differential on $E^1$ is precisely the map induced by the diagonal inclusion $M \overset{\Delta}{\into} M^2$, thus $[M]$ above has the desired boundary.

To summarize, while the primary stability involves multiplication by $[M]\in \Ho^{BM}_d(M) \cong E^1_{0,d}[1]$ -- a dual to the class of any point in $M$, there is a secondary stabilization map given by multiplication with $[M]\in E^1_{1,d}[2]$ -- this one dual to introducing a pair of closely orbiting points.
To understand the exact relation with Miller-Wilson's stabilization operators, consider multiplication maps by two Poincar\'e dual classes. These operations, say $T$ on $\Ho_i$ and $T^{BM}$ on $\Ho^{BM}_{\dim-i}$, have the property that on the intersection pairing $\Ho_i\otimes \Ho^{BM}_{\dim-i}$ one gets
\[
\langle T\alpha, T^{BM}\beta \rangle = \langle \alpha, \beta \rangle.
\]
Therefore the dual to each map is a right-inverse (a retraction) of the other, e.g. $T^*\circ T^{BM}= Id$. In particular, both maps $T$ and $T^{BM}$ are injective, and if either one is surjective so will be the other.
With this in mind, the argument above shows that the action of the GTCA $A^2_d=\Ind_{(2)}^{\FB[]}\Ho^{BM}_d(M)$ is related with Miller-Wilson's secondary stability map in this way, and this relation makes Miller-Wilson's stability results equivalent to ours whenever both maps are defined.

Now, the above description assumed that we had a manifold with trivial group action; consider next a $G$-space $X$ under Hypothesis \ref{hyp:vanishing_diagonal}.
Let $Z$ be a $(G\times \sym[2])$-stable complement to the image $\left( \Ho^{BM}_{d+1}(X^2) \to \Ho^{BM}_{d+1}(\Conf^2(X)) \right)$.  As is evident from the long exact sequence of the pair
$\Delta\subset X^2$ mentioned in the previous few paragraphs, $Z$ is equivariantly isomorphic to $\Ho^{BM}_d(X)$. The discussion above also shows that when $X$ is a manifold, $Z$ specializes to a Poincar\'{e} dual space to the classes of closely orbiting pairs.

The open inclusion $\Conf^2_G(X,T) \into \Conf^2(X)=\Conf^2_{\textbf{1}}(X,\emptyset)$ induces a restriction map $\Ho^{BM}_{d+1}(\Conf^2_{\textbf{1}}(X,\emptyset)) \to \Ho^{BM}_{d+1}(\Conf^2_G(X,T))$ that is clearly compatible with the collision filtrations and the description of the associated spectral sequences. In particular, the induced restriction map on $E^1_{1,d}[2]$ becomes
\[
\Ho^{BM}_d(X) = \Ind_{\textbf{1}\times \sym[2]}^{\textbf{1}^2\ltimes \sym[2]}\Ho^{BM}_d(X) \to \Ind_{G\times \sym[2]}^{G^2\ltimes \sym[2]}\Ho^{BM}_d(X).
\]
This identifies the restriction of $Z\leq \Ho^{BM}_{d+1}(\Conf^2(X))$ to $\Conf^2_G(X,T)$ with the subspace $\Ho^{BM}_d(X)\leq E^1_{d,1}[2]$ from which the term is induced.

\begin{theorem}[\textbf{Stabilization by orbiting pair analogue}]\label{thm:orbiting_pair}
Let $X$ be a $G$-space of homological dimension $d\geq 2$ satisfying the additional Hypothesis \ref{hyp:vanishing_diagonal}, and let 
$Z$ be the restriction to $\Ho^{BM}_{d+1}(\Conf^2_G(X,T))$ of a $(G\times\sym[2])$-stable complement to the image of the map
$\left( \Ho^{BM}_{d+1}(X^2) \to \Ho^{BM}_{d+1}(\Conf^2(X)) \right)$.
The cross product
\[
Z \otimes \Ho^{BM}_j(\Conf^m_G(X,T)) \to \Ho^{BM}_{(d+1)+j}(\Conf^{m+2}_G(X,T))
\]
makes the homology into modules
\begin{equation}\label{eq:module_secondary}
\Ho^{BM}_{(d+1)\bullet-i}(\Conf^{2\bullet}_G(X,T)) \text{ and } \Ho^{BM}_{(d+1)\bullet-j}(\Conf^{2\bullet+j}_G(X,T))
\end{equation}
over the GTCA $\Ind_{G\times \sym[2]}^{\FB} Z$ for the various values of $j \in \ZZ$ (possibly $<0$).

If $X$ is $i$-acyclic then every homology module under iterated multiplication by $Z$ is free and finitely generated by classes coming from configurations with $\bullet\leq 2j$ modulo images of the cross product actions \[\Ho^{BM}_{d-i}(X)\otimes \Ho^{BM}_*(\Conf^\bullet_G(X,T)) \to \Ho^{BM}_{*+d-i}(\Conf^{\bullet+1}_G(X,T))\] for all $ 0\leq i \leq \frac{d-1}{2}$.
\end{theorem}

\begin{proof}
By the iterative process described in Theorem \ref{thm:E1_finite_generation} one deduces that after factoring out the cross product action by $\Ho^{BM}_{d-k}(X)$ for $0\leq k\leq \frac{d-1}{2}$, the multiplication by $\Ho^{BM}_d(X)\leq E^1_{1,d}[2]$ makes the $E^1$ page into a collection of finitely generated and free modules over the GTCA $\Ind^{\FB}_{G\times \sym[2]} \Ho^{BM}_d(X)$. Furthermore, the bounds on generators given in Theorem \ref{thm:E1_finite_generation} specialize for the modules in \eqref{eq:module_secondary} to $2j$, as one sees by looking at the generation locus in Figure \ref{fig:genlocus} and observing that the minimal difference of taxi-cab norms between $(\frac{1}{2},\frac{d}{2})$ and any point below it in $\Gen(E)$ is $\frac{1}{2}$.

Now, our choice of $Z\leq \Ho^{BM}_{d+1}(\Conf^2(X))$ is naturally identified with the generating subspace $\Ho^{BM}_d(X) \leq E^1_{d,1}[2]$ after passing to the associated graded of the collision filtration. In particular, the cross product with $Z$ coincides with the  multiplication of the previous paragraph, thus making $E^1$ into finitely generated modules over $\Ind_{G\times \sym[2]}^{\FB}Z$. 

When $X$ is an $i$-acyclic space, Remark \ref{rmk:differentials} states that all differentials of the collision spectral sequence vanish. Thus the finite generation and bounds on the $E^1$-page imply the same bounds on homology, thus completing the proof for such spaces.
\end{proof}

For a space $X$ not satisfying the $i$-acyclicity assumption the unknown differentials pose a challenge. However, when $\dim \Ho^{BM}_d(X) = 1$, the  Noetherianity of TCAs could allow us to get around this difficulty. Unfortunately, this theory is only well-developed for the case of a trivial group $G$ and for TCAs generated by a single element over a field of characteristic $0$; hence we consider only this case below. Suppose therefore that $\dim \Ho^{BM}_d(X) = 1$. The TCA  $A^2_d$ acting on homology is the free graded-commutative TCA generated by a one-dimensional $\sym[2]$-representation. Such TCAs are exactly one of following known in the literature as $\operatorname{Sym}(\operatorname{Sym}^2(\mathbb{C}))$, $\operatorname{Sym}(\operatorname{\Lambda}^2(\mathbb{C}))$, $\operatorname{\Lambda}(\operatorname{Sym}^2(\mathbb{C}))$, and $\operatorname{\Lambda}(\operatorname{\Lambda}^2(\mathbb{C}))$. In \cite{NSS1,NSS2}, Nagpal, Sam, and Snowden prove that over a field of characteristic $0$, all four TCAs possess the Noetherian property. Since after factoring out the actions of the TCAs $A^1_{d-i}$ for $0\leq i\leq \frac{d-1}{2}$ every module in the $E^1$ page is finitely generated, the same property extends to their subquotient modules. To get the same result applied to homology one would have to understand the effect that factoring out stabilization actions has at the level of homology, and whether finite generation persists through such operations.
\begin{conjecture}
Theorem \ref{thm:orbiting_pair} holds without the assumption that $X$ is $i$-acyclic, though now with non-free modules over $\Ind^{\FB[]}_{G\times \sym[2]}Z$ and with worse bounds on where generators appear.
\end{conjecture}

\subsubsection{Relation to previous work}
There are several differences between our Theorem \ref{thm:orbiting_pair} and Miller--Wilson secondary stability \cite{MW}. Perhaps most importantly, Miller-Wilson give a concrete stable range beyond which stability is known to hold (controlling where generators and relations may appear) without assuming any differentials vanish. In our analysis beyond primary stability we only consider $i$-acyclic spaces, for which one encounters no differentials.
However, the explicit bounds on generators and freeness of the $E^1$ page in Theorem \ref{thm:E1_finite_generation} provides the necessary input to bound stable ranges using common representation stability techniques.

Putting stable range calculations aside, we obtain finite generation without handling the non-trivial combinatorics of the complex of injective words, which Miller--Wilson had to understand to get their finite generation results.
Secondly, our theorem is only sensitive to the proper homotopy type of $X$, via $\Ho^{BM}_*(X)$ and its diagonal maps, while the previous work asked that $X$ specifically be a manifold with boundary.

Moreover, Miller-Wilson always stabilize by a $1$-cycle of orbiting points, regardless of the dimension of the manifold $X$. This had the consequence that their stabilization map was trivially $0$ when $\dim(X)\geq 3$. From this perspective, our stabilization map, which uses a $(d-1)$-cycle of orbiting points, 
is more natural and meaningful for manifolds of any dimension.

But with these remarks, also note that Miller--Wilson proved their stability result without the assumption that $\Delta_* = 0$. We believe that a more refined version of Theorem \ref{thm:orbiting_pair} is possible, replacing $Z$ with $\ker \Delta_*$, but the homological algebra involved would pose too great of a distraction at this point of this already lengthy document. It is our hope to address the general case in a sequel.

Another expected extension of this result for non-$i$-acyclic spaces is when $G$ is not the trivial group. With the current available technology of representation stability, such a result is out of reach. But the product factorization in Theorem \ref{thm:ss_factorization} suggests that one could hope to bootstrap the $G=1$ case up to a general $G$ without much trouble. Such an extension is also a potential subject of a sequel.

\subsection{Additional forms of stability} \label{subsec:new_stability}
The geometric criteria for finite generation in \S\ref{subsec:generation_locus} 
reveal two new types stabilization actions that are associated with finite generation.

\subsubsection{Bottom-left corner stabilization}
In Corollary \ref{cor:i-acyclic}, we described a succession of stabilization actions that arise from lines touching the top of the generation locus in Figure \ref{fig:genlocus}. An analogous process proceeds when considering lines touching the bottom of $\Gen(E)$.
\begin{theorem}[\textbf{Bottom-corner stabilization process}]
Assume that $X$ is an almost free $G$-space following Conventions \ref{conventions} with $\dim \Ho_i^{BM}(X)<\infty$ for each $i$, and let $T\subset X$ be a finite $G$-invariant subset.
Then the cross product \[
\Ho^{BM}_0(X)\otimes \Ho^{BM}_i(\Conf^n_G(X,T)) \to \Ho^{BM}_i(\Conf^{n+1}_G(X,T))
\]
endows $\Ho^{BM}_i(\Conf^\bullet_G(X,T))$ with the structure of a finitely generated module over the GTCA $\Ind \Ho^{BM}_0(X)$.

If $X$ is a $d$-dimensional manifold, under Poincar\'{e} duality this translates to a finitely-generated module structure in every fixed codimension
$\Ho^{d\bullet-i}(\Conf^\bullet_G(X,T))$ under cross product by the volume forms in $\Ho^d(X)$.

Proceeding to secondary operations, in the case of $i$-acyclic spaces the consecutive cross products by $\Ho^{BM}_k(X)$ with $k=1,2,\ldots$ endows the homology $\Ho^{BM}_{k\bullet-i}(\Conf^\bullet_G(X,T))$ with a filtered-finite generation structure relative to the collision filtration after factoring out the multiplication by $\Ho^{BM}_j(X)$ for $j<k$.
\end{theorem}
\begin{proof}
For the multiplication by $\Ho_0^{BM}(X)$ the same proof as for Theorem \ref{thm:E1_finite_generation} applies by replacing the slope $(-1)$ line touching $\Gen(E)$ from above with one touching from below.

Then for secondary and higher stabilization, fix $k\geq 1$. Factoring out the multiplication by $\Ho_j^{BM}(X)$ for $j<k$ removes the corresponding points $(0,j)$ from $\Gen(E)$, thus forming an isolated corner at $(0,k)$. The criterion of slope $\neq -1$ in Lemma \ref{lem:filtered_fg} states that multiplication by $\Ho^{BM}_k(X)$ endows the remaining diagonals in the $E^1$ page with the structure of a bounded-finitely generated module. From that point, the same Noetherianity argument as in \ref{thm:primary-finite-generation} shows that the homology is also bounded-finitely generated after factoring out the previous multiplication actions.
\end{proof}

\subsubsection{Truncated-finite generation}
The complementary notion to bounded finite generation is finite generation after truncation. Whenever the stabilization process of Theorem \ref{thm:E1_finite_generation} failed to produce finite generation (when there were multiple points on of $\Gen(E)$ of maximal taxi-cab norm) we have chosen to stabilize with the left-most point, thus giving bounded-finite generation by Lemma \ref{lem:filtered_fg}. We could have alternatively chosen to continue the process by picking the right-most point, giving instead truncated-finite generation. This approach gives a completely analogous version of Theorem \ref{thm:E1_finite_generation}, and in fact one could mix the two versions freely.

\begin{proposition}[\textbf{Stabilization with locally finite generation}]
Consider the setup of Theorem \ref{thm:E1_finite_generation} and the iterative stabilization process described therein. If at any iterate one encounters multiple points of maximal taxi-cab norm, they may pick the right-most point instead of the instruction in Step (1.) to pick the one furthest to the left. Under such a choice every diagonal of $E_{*,*}$ forms a free module that is finitely generated after truncating the (shifted) collision filtration.

To understand the module of (free) generators, proceed to Step (2.) of Theorem \ref{thm:E1_finite_generation} without any further adjustments.
\end{proposition}
\begin{proof}
An analogous argument as in the proof of Theorem \ref{thm:E1_finite_generation} works, but here one will consider a line of slope $-1-\epsilon$ separating the right-most point of maximal taxi-cab norm from the rest. Thus by Lemma \ref{lem:filtered_fg} truncated-finite generation follows.
\end{proof}

Even more, after the primary stabilization action has been factored out, one could have proceeded to stabilize with the ``orbiting pair" action mentioned in Theorem \ref{thm:orbiting_pair} to get the same theorem but with truncated-finite generation taking the place of factoring out the cross products by $\Ho^{BM}_{d-k}(X)$. This is a direct consequence of Lemma \ref{lem:filtered_fg}, after observing that without the top corner in Figure \ref{fig:genlocus} a new corner is formed at $\left(\frac{1}{2}, \frac{d}{2}\right)$.

There is however a substantial reason to prefer the version of Theorem \ref{thm:E1_finite_generation} as presented above: GTCAs generated in low degrees are much better understood compared to general ones. For example, any GTCA that is finitely generated in degree $1$ is Noetherian, thus giving hope that the finite-generation results of Corollary \ref{cor:i-acyclic} would apply to general spaces, while for TCAs generated in degree $\geq 3$ not much is known. The generation locus in Example \ref{ex:genlocus} has the property that points further to the left come from configurations of smaller numbers of points, and thus give better control over differentials by the above comment.

\bibliographystyle{alpha}
\bibliography{paper}

\end{document}